\documentclass[11pt]{article}

\usepackage[utf8]{inputenc}
\usepackage[T1]{fontenc}

\usepackage{amssymb}
\usepackage{bbm}
\usepackage{amsmath}
\usepackage[left=3cm,right=3cm,top=3cm,bottom=3cm]{geometry}
\usepackage{listings}
\usepackage{graphicx}
\usepackage{mathtools}
\usepackage{color}
\usepackage{esint}
\usepackage{enumitem}
\usepackage{tikz}
\usepackage{wrapfig}
\usepackage[ntheorem]{empheq}
\usepackage[thref,hyperref,thmmarks,amsmath]{ntheorem}
\usepackage{hyperref}
\usepackage{framed}
\usepackage{unnumberedtotoc}
\usepackage{setspace}
\usepackage{chngcntr}

\def\N{\mathbb{N}}

\def\R{\mathbb{R}}
\def\Q{\mathbb{Q}}
\def\C{\mathrm{C}}
\def\H{\mathcal{H}}
\def\M{\mathcal{M}}
\def\B{\mathbb{B}}
\def\P{\mathcal{P}}
\def\F{\mathcal{F}}

\def\1{\mathbbm{1}}
\def\L{\mathcal{L}}
\def\G{\mathbb{G}}
\def\E{\mathbb{E}}
\def\A{\mathbb{A}}
\def\D{\mathcal{D}}
\def\F{\mathcal{F}}

\def\BC{\mathrm{BC}}

\def\eps{\varepsilon}

\def\supp{\mathrm{supp}}

\newcommand{\mres}{\mathbin{\vrule height 1.6ex depth 0pt width 0.13ex\vrule height 0.13ex depth 0pt width 1.3ex}}

\newcommand\eqrefp[1]{eq. \eqref{#1}, p. \pageref{#1}}
\newcommand\refp[1]{\thref{#1}, p. \pageref{#1}}
\newcommand\citep[2]{\cite{#1}, p. #2}

\newtheorem{theorem}{Theorem}
\newtheorem{proposition}[theorem]{Proposition}
\newtheorem{lemma}[theorem]{Lemma}
\newtheorem{corollary}[theorem]{Corollary}
\newtheorem{definition}[theorem]{Definition}

\newtheorem{remark}[theorem]{Remark}

\newtheorem*{proof2}{\textbf{Proof.}}

\newenvironment{proof}{\begin{proof2}}{ {\hspace*{0pt}\hfill $\square$} \end{proof2}}

\counterwithin{theorem}{section}
\counterwithin{exercise}{section}
\counterwithin{equation}{section}

\mathtoolsset{showonlyrefs}

\begin{document}


\title{\textbf{ Generalised Young Measures\\ and\\ characterisation of gradient Young Measures }}
\author{Tommaso Seneci }
\date{}
\maketitle
\thispagestyle{empty}

\setlength{\parindent}{0pt}



\begin{abstract}
Given a function $f\in C(\R^d)$ of linear growth, we give a new way of representing accumulation points of
\begin{equation}
\int_\Omega f(v_i(z))d\mu(z),
\end{equation}
where $\mu\in \M^+(\Omega)$, and $(v_i)_{i\in \N}\subset L^1(\Omega,\mu)$ is norm bounded. We call such representations "generalised Young Measures". With the help of the new representations, we then characterise these limits when they are generated by gradients, i.e. when $v_i = Du_i$ for $u_i\in W^{1,1}(\Omega,\R^m)$, via a set of integral inequalities.
\end{abstract}

\newpage
\tableofcontents

\newpage
\section{Intro}

\subsection{Terminology and symbols}

\begin{itemize}
\item For a vector $v\in \R^m$, we write $|v| = \sqrt{\sum_{i=1}^m v_i^2}$ otherwise specified.
\item Given a function $f:X\to Z$ and $W$ an arbitrary set, we call
\begin{align}
graph(f) \equiv & graph_X(f) = \{ (x,f(x))\in X\times Z \text{ such that } x\in X\}, \\
& graph_{X\times W}(f) = \{ (x,w,f(x))\in X\times W\times Z \text{ such that } x\in X, w\in W \}.
\end{align}
\item We say that a function $f:\R^d\to \R$ has $p$ growth if there is $C>0$ such that
\begin{align}
|f(z)| \leq C(1+|z|^p).
\end{align}
\item The identity matrix is indicated by $\1$, or $\1_d\in \R^{d\times d}$ if we need to specify the dimension.
\item The Lebesgue measure is indicated by $dx$, or $\L^n$ depending on the context. The set of finite Borel d-vector measures is indicated by $\M(\Omega,\R^d)$. For $E\subset \M(\Omega)$, $E^+$ is the set of positive Borel measures that belong to $E$.
\item For a measure $\mu\in \M(\Omega)^+$, we write $L^p(\Omega,\mu,\R^d)$ to mean the space of $\mu$-measurable functions $f \colon \Omega\to \R^m$ such that
\begin{align}
\int_\Omega |f(x)|^pd\mu(x) <+ \infty.
\end{align}
\item For $\mu\in \M(\Omega,\R^d)$, we write its restriction to a $\mu$-measurable set $E\subset \Omega$
\begin{align}
\mu\mres E \colon A \text{ Borel set} \mapsto \mu(E\cap A).
\end{align}
\item If $\mu\in \M(\Omega)$ and $f\in \L(\Omega,\mu,\R^d)$, we call $fd\mu$ the measure in $\M(\Omega,\R^d)$ defined by
\begin{align}
U \mapsto \int_U fd\mu,
\end{align}
where $U$ runs through all \text{$\mu$-measurable sets}.
\item If $X$ is any space, the Dirac delta is indicated, for $x\in X$, by
\begin{align}
\int_X f(y) d\delta_x(y) = f(x),
\end{align}
where $f\colon X\to \R^d$ is an arbitrary function.
\item Let $X$ be a metric space, $\mu\in \M^+(\Omega)$ and $(f_j)_{j\in \N} \subset L^p(\Omega,\mu,\R^d)$. We say that the sequence $(f_j)_j$ is p-equi-integrable (or simply equi-integrable if $p$ is clear from the context) if it is norm bounded and
\begin{align}
\lim_{k\uparrow \infty} \ \sup_{j\in \N} \ \int_{|f_j|^p>k} |f_j|^pd\mu = 0.
\end{align}
\item The set of test functions is
\begin{align}
\D(\Omega,\R^d) = C^\infty_c(\Omega,\R^d) = \{ f:\Omega\to \R^d: f \text{ is infinitely differentiable and has compact support} \}.
\end{align}
We do not insist on the topology this space is endowed with, as it is standard and nowhere used in the work.
\item For the derivative of a function $u\in L^1(\Omega,\R^m)$ we mean the matrix-valued distribution $Du = [\partial_j u_i]_{i,j} \in (\D(\Omega,\R^m)^*)^n$ such that
\begin{align}
\int_\Omega u_i \partial_j \phi dx = - \langle \partial_j u_i,\phi\rangle.
\end{align}
\item The Sobolev space of functions with integrable derivatives is
\begin{align}
W^{1,1}(\Omega,\R^m) = \{ u\in L^1(\Omega,\R^m) : Du\in L^1(\Omega,\R^{m\times n}) \}.
\end{align}
\item The set of functions of bounded variation is
\begin{align}
BV(\Omega,\R^m) = \{ u\in L^1(\Omega,\R^m) : Du\in \M(\Omega,\R^{m\times n}) \}.
\end{align}
\item For a function $u\in BV(\Omega,\R^m)$ we can write
\begin{align}
Du = \nabla u d\L^n\mres \Omega + D^s u, \quad D^s u = D^j u\mres J_u + D^c u
\end{align}
where $\nabla u d\L^n$ is the absolutely continuous part, $D^j u$ is the jump part concentrated on a $n-1$ rectifiable set $J_u$, and $D^s$ is the Cantor part, which is absolutely continuous with respect to $\H^{n-1}$.
\item For a function $U\in BV(\Omega,\R^m)$, we call
\begin{align}
BV_U(\Omega,\R^m) = \left\{ \begin{matrix}
u\in BV(\Omega,\R^m): \text{there is a sequence } u_j\in \D(\Omega,\R^m) \\
\quad \text{ such that } u_j \xrightarrow{\text{weak* in BV}} u-U
\end{matrix} \right\}.
\end{align}
\item The set of special functions of bounded variation is
\begin{align}
SBV(\Omega,\R^m) = \{ u\in BV(\Omega,\R^m): D^s u = D^j u \text{, or equivalently } D^c u=0\}.
\end{align}
\end{itemize}

\newpage

\subsection{Introduction}

Young Measures were first introduced by Young in \cite{young1937generalized} to study the minima of integral energies of the form
\begin{equation}
\inf\left\{ \int_0^1 f(u(t),u'(t))dt: u\in C^1([0,1]), u(0)=a, u(1)=b, \|u'\|_\infty \leq K \right\}.
\end{equation}
The author wanted to understand what conditions on $f$ would guarantee the existence of a minimizing curve $u(t)$. Young had the intuition that, for an extremely general class of functions $f$, minimising sequences always converge to a "generalised" curve $t\mapsto (u(t),\nu_t)$ where $\nu$ is a probability measure on the image of $f$. This translates to the following equality
\begin{equation}
\inf_{u(0)=1,u(1)=b} \int_0^1 f(u(t),u'(t))dt = \lim_j \int_0^1 f(u_j(t),u_j'(t))dt = \int_0^1 \int_\R f(u(t),y)d\nu_t(y) dt.
\end{equation}
So the question of the existence of a minimiser can be reformulated as to whether such objects are gradients of a curve or not. $\nu_t$ might fail to be a gradient when the minimizing sequence oscillates. \\

Young's original work focused on the case $n=1$ and was carried out via functional analytic methods. This approach was later extended in \cite{ball1989version,berliocchi1973integrandes} to higher dimensions. We call these generalised functions "oscillation Young Measures". The method developed by Young is not powerful enough to tackle problems arising in modern mathematics, as it can only handle sequences $(v_j)_{j\in\N}$ that are bounded in $L^\infty$ rather than in some Lebesgue space $L^p$. The first attempt to well represent generalised limits of integrable functions is due to DiPerna and Majda, \cite{diperna1987oscillations}. Functions $v_j\colon \Omega\to \R^d$ are seen as Dirac deltas on the product space $\Omega\times \R^d$, which is subsequently compactified. An accumulation point, in the sense of these new generalised functions, is then found. Such an accumulation point is a measure defined on an abstract compactification of $\Omega\times \R^d$, and as such, it is not clear how to represent it in the original, non-compact, space. In \cite{alibert1997non}, an explicit formula for such accumulation points was obtained for a class of integrands that grow "nicely" infinity.

In what follows, we give a general formula for describing Young Measures for a large class of integrands. The construction of Young Measures follows mainly the work by DiPerna and Majda \cite{diperna1987oscillations} and lecture notes taken from a class given by Kristensen \cite{kristensen2015young}, see also \cite{rindler2018calculus}, chapter 12. This generalisation is based on the canonical way of constructing Hausdorff compactifications starting from continuous functions, see \cite{chandler1976hausdorff}. A small reduction lemma gives a clearer, and somehow geometrical interpretation of such limits. This formula captures oscillations at infinity, which are now let occur. We also prove a few structure theorems that relate different compactifications and Young Measures representations to each other. \\

This generalisation of Young Measures is then applied to study extensions and variations - within the class of functions of bounded variations $BV(\Omega,\R^m)$ - of energies that depend on gradients
\begin{equation} \label{Introduction:eq:functional_gradient}
u\mapsto \int_\Omega f(Du(x))dx,
\end{equation}
where $u\in \D(\Omega,\R^d), \Omega\subset \R^n$ is a bounded domain, and $f\in C(\R^{m\times n})$ has linear growth. Given any such $f$, there is no way to extend \eqref{Introduction:eq:functional_gradient} to the class BV so that such extension is continuous with respect to sequential weak* convergence in $C_0(\Omega,\R^{m\times n})^*$. We can however find an extension which is lower semi-continuous for certain $f$s. In \cite{morrey1952}, Morrey established the equivalence of lower semi-continuity of \eqref{Introduction:eq:functional_gradient} to a condition named "quasi-convexity", which can be written as a Jensen-type inequality
\begin{equation} \label{Introduction:eq:quasi_convexity}
\int_\Omega f(z+D\phi(x))dx \geq |\Omega|f(z) \quad \forall \phi\in \D(\Omega).
\end{equation}
The original result by Morrey works in the setting of weak* convergence in $W^{1,\infty}(\Omega,\R^m)$, and it was subsequently extended to the case $W^{1,p}(\Omega,\R^m), 1\leq p<\infty$ and weak convergence in \cite{acerbi1984semicontinuity}, for positive integrands. As for signed integrands, the same result was proven in \cite{ball1990lower} and it is one of the first examples where Young Measures are employed for proving lower semi-continuity in the space of gradients. To be more specific, \eqref{Introduction:eq:quasi_convexity} can be rephrased as a Jensen-type inequality for measures of the form
\begin{equation} \label{pullback_gradients}
\{ \nu_x: \nu_x = D\phi(x)_\# d\L^n\mres \Omega, \phi\in C_c^\infty(\Omega, \R^m) \},
\end{equation}
where $\nu_x$ acts on $f$ in the following way:
\begin{equation}
\int_\Omega f(z+D\phi(x))dx = \int_\Omega \int_{\R^{m\times n}} f(z+w) d\nu_x(w) dx \equiv \int_\Omega \langle \nu_x,f\rangle dx.
\end{equation}
The lower semi-continuity of \eqref{Introduction:eq:functional_gradient} becomes a functional analytic inequality of the form
\begin{align}
\int_\Omega \langle \nu_x,f\rangle dx \geq \int_\Omega f(Du(x))dx \quad \text{and} \quad Du(x) = \int_{\R^{m\times n}} zd\nu_x,
\end{align}
and in this case we call $x\mapsto \nu_x$ a "Gradient Young Measure". This class can be seen as the closure of the set \eqref{pullback_gradients} in the weak* topology of measures over the graph of $f$. The opposite is also true and was proven for the first time in \cite{kinderlehrer1991characterizations,kinderlehrer1994gradient}, i.e. every measure-valued function $x\mapsto \nu_x$, for which a Jensen's type inequality holds against quasi-convex functions of suitable growth, is the limit of a sequence of gradients. \\

The aforementioned results hold in the setting of weak convergence in $W^{1,p}, 1\leq p<\infty$ and weak* convergence in $W^{1,\infty}$. This is a natural condition to assume when $p>1$, but not when $p=1$, as the Lebesgue space $L^1(\Omega,\L^n)$ is not reflexive. In particular, a bounded sequence in $L^1(\Omega,\L^n)$ can concentrate and converge to measures that are singular with respect to the Lebesgue measure. In terms of gradients, the closure of $W^{1,1}(\Omega,\R^m)$ so that its unit ball is weak* compact is the set of functions of bounded variations $BV(\Omega,\R^m)$, precisely the set of functions whose derivatives are measures. This concentration phenomenon is exclusive of the case $p=1$, and so regards integrands that have linear growth at infinity. It turns out, as proven in \cite{ambrosio1992relaxation}, that when $f$ has linear growth and it's quasi-convex, the integral functional $u\mapsto \int_\Omega f(\nabla u)dx,f\geq 0$ is still lower semi-continuous in $BV(\Omega,\R^m)$ with respect to the weak* topology, but there is a deficit of mass when gradients concentrate. Letting $f$ be so that
\begin{align} \label{Introduction:f_infinity}
f^\infty(z) = \lim_{t\to \infty,z_n\to z} \frac{f(tz_n)}{t}
\end{align}
exists for all $z_n\to z, t\to \infty$, the lower semi-continuous envelope of \eqref{Introduction:eq:functional_gradient} in the space $BV(\Omega,\R^m)$, with respect to sequential weak* convergence, is, for $f$ non-negative,
\begin{equation}
u\mapsto \int_\Omega f(\nabla u(x))dx + f^\infty \left( \frac{D^s u}{|D^s u|}(x) \right) d|D^s u|(x).
\end{equation}
In this case, a Young Measure formulation of the Jensen's-type inequality \eqref{Introduction:eq:quasi_convexity} has to take into account the singular part of $Du$. In the spirit of the previous results, one is tempted to prove a duality-type characterisation of Young Measures with concentrations and quasi-convex functions. Differently from the case without concentration, here we assumed $f^\infty$ to exist as in \eqrefp{Introduction:f_infinity}. However, as shown in \cite{muller1992quasiconvex}, quasi-convex functions can oscillate at infinity, meaning that $f^\infty(z)$ may not exist for some $z\in \R^m$. This suggests that to obtain a Jensen-type inequality and characterisation result for gradient Young Measure in the case $p=1$, it is necessary to specify a compactification at infinity. 

The characterisation for gradient Young Measures when $p=1$ has been already obtained on the so-called "sphere compactification" - functions for which $f^\infty(z)$ exists for all $z$ - see \cite{kristensen2010characterization,kristensen2010relaxation}.  After showing that the class of quasi-convex functions of linear growth is too big to be included within any separable compactification, we reprove the characterisation result for gradient Young Measures on separable compactifications of quasi-convex functions. This restricts the number of quasi-convex functions to be considered at once. However, it is also inevitable because a compactification containing all quasi-convex functions would be so big that its topology would fail to be metrisable and separable. \\

\newpage
\section{Generalised Young Measures on separable compactifications} \label{chap1}

In this section, we construct generalised Young Measures and provide a new geometric representation. Concentration is let "oscillate with different amplitudes at infinity". To do so, we embed a space of functions into a bigger compact set and subsequently use the theory of Hausdorff compactifications.

\subsection{Generalised Young Measures as generalised objects}

Generalised Young Measures are objects that were known to exist since Majda and Di Perna \cite{diperna1987oscillations}, and have been used in a few instances, see for example \cite{fonseca2010oscillations,kruvzik1996explicit}. However, their existence per se does not give enough clarity on their properties, and so makes it hard to work with such objects.

We give a new interpretation and geometric representation that better captures oscillation and concentration effects that occur in limits of the form
\begin{align}
\lim_j \int_\Omega f(v_j(x))d\mu(x),
\end{align}
where $v_j\in L^p(\Omega,\mu,\R^d)$ is a norm-bounded sequence and $f\in C(\R^d)$ has p-growth. Under these assumptions, it is easy to see that, up to a subsequence, $f\circ v_j$ converges to a measure $\nu$; mathematically this means that
\begin{align}
\int_\Omega f(v_j(x)) \phi(x)d\mu(x) \to \int_\Omega \phi(x)d\nu(x)
\end{align}
for all $\phi\in C_0(\Omega)$. It's clear that $\nu=\nu(f)$ is a linear function of $f$. What is not clear is how such dependence can be represented in terms of $\mu$ and $f$. Without a clear representation, it is not possible to set up a system of calculus. This section is dedicated to working out a geometric interpretation of the relation between $\nu$ and $f$, which we will then call Young Measure. We will mainly concentrate on the more interesting and harder case of $p=1$ and $v_j =Du_j$ gradients, where concentration effects create rather complicated structures, and cannot in general be separated from oscillation.

\begin{definition}
A function $f:\R^d \to \R$ is said to have $p$-growth if there is a constant $C\geq 0$ such that
\begin{align}
|f(z)|\leq C(1+|z|^p) \quad \forall z\in \R^m.
\end{align}
When $p=1$, we say that such functions have linear growth.
\end{definition} 

Before proceeding with formal definitions, we give a heuristic interpretation of Young Measures. When $v_j\to v$ strongly in  $L^1(\mu)$ then the limit Young Measure is trivial, meaning that
\begin{align}
\int_\Omega f(v_j(x)) \phi(x) d\mu(x) \to \int_\Omega f(v(x)) \phi(x) d\mu(x)
\end{align}
for all $f\in C(\R^d)$  of linear growth and $\phi\in C_0(\Omega)$. This is a simple consequence of the Vitali convergence \refp{Appendix:Vitali_conv} (or the generalised dominated convergence theorem).

When strong $L^1$ convergence fails, only two things can go wrong:
\begin{itemize}
\item \textbf{oscillation} - $v_j$ oscillates around $\mu$-almost every point $x\in E\subset \Omega$ with $\mu(E)>0$, and generates a probability distribution on the target space
\begin{align}
f(v_j(x)) \leadsto \langle \nu_x,f\rangle = \int_{\R^d} f(z)d\nu_x;
\end{align}
\item \textbf{concentration} - $|v_j|$ concentrates to a measure $0\neq \lambda\in \M^+(\overline{\Omega})$ - equivalently $(x,v_j(x))$ concentrates to the boundary of some compactification of $\Omega\times \R^d$. Around $\lambda$-almost every point, $v_j(x)$ goes to infinity and its "support" collapse to 0. That is to say, 
\begin{align}
f(v_j(x)) \ \leadsto \ \left\{ \frac{f(z)}{|z|} \text{ with } |z|\gg 0 \right\} \ \sim \ \int_{\partial K} f^\infty(w)d\nu_x^\infty(w),
\end{align}
where $K$ is some compactification containing $\R^d$ that extends $f$ to $f^\infty$ on the remainder of $\R^d$ within $K$.
\end{itemize}

\subsubsection{Parametrized measures}

In order to construct Young Measures, we regard functions as maps from a domain $\Omega$ into the set of probability measures over a target space $\R^d$. Ordinary functions $f\colon \Omega\to \R^d,x\mapsto f(x)$ are embedded into maps $\Omega\to \M^+_1(\R^d), x \mapsto \delta_{f(x)}$. Preliminary to the construction, we introduce two basic concepts that are at the core of this theory. 

\begin{definition}
Let $X$ and $Z$ be locally compact, separable metric spaces and $\lambda\in \M^+(X)$. A map $\nu\colon X\to \M^+(Z)$ is said to be $\lambda$-measurable if for each $\phi\in C_0(Z)$ the function $x\mapsto \langle \nu(x),\phi\rangle$ is $\lambda$-measurable.
\end{definition}

We shall often write the measure-valued map $\nu$ as a \textit{parametrized measure} $(\nu_x)_{x\in X}$, where $\nu_x\colon= \nu(x)$. Given a measure $\nu$ on a product space $X\times Z$, it can always be decomposed as a product of its projection onto $X$ and its cross section on $Z$.

\begin{theorem}[Disintegration of measures] \label{s4:disintegration_theorem}
Let $X$ and $Z$ be compact metric spaces and denote by $\pi\colon X\times Y \to Z$ the projection mapping onto the first coordinate $\pi(x,y)=x$. For $u\in \M^+(X\times Z)$ and $\lambda = \pi_\# \nu\in \M^+(X)$ (the pushforward of $\nu$ via $\pi$) there exists a unique $\lambda$-measurable parametrized measure $(\eta_x)_{x\in X},\eta_x\in \M^+_1(Z)$ such that for all $\phi\in C(X),\psi \in C(Z)$ we have
\begin{align}
\langle \nu,\phi\otimes \psi \rangle = \int_X \langle \eta_x,\psi\rangle \phi(x)d\lambda(x) = \int_X \int_Z 	\psi(z) d\eta_x(z) \phi(x)d\lambda(x).
\end{align}
\end{theorem}

For a proof of this result see \citep{ambrosio2000functions}{57}. In this case we write
\begin{align}
\nu = \eta_x d\lambda.
\end{align}

\subsubsection{Generalized Young measures}
	
In what follows, we show how to obtain a good representation of Young Measures on general compactifications. The procedure is adapted from some lecture notes taken from a homonym course given by Jan Kristensen at the University of Oxford, \cite{kristensen2015young}. Some of the results can also be found in \cite{rindler2018calculus}, chapter 12, where they are only proven on the sphere compactification.

Throughout this section, $\Omega\subset \R^n$ is open and bounded, $\mu\in \M^+(\Omega)$ and $(v_j)_j \subset L^p(\Omega,\mu,\R^d)$ is a bounded sequence, $1\leq p<\infty$. Assume
\begin{align}
v_j \rightharpoonup v \text{ in } L^p \text{ when } \ 1<p<\infty \quad \text{ or } \quad v_j \rightharpoonup^* v \text{ in } C_0(\Omega,\R^d)^* \text{ when } \ p=1.
\end{align}
Given a continuous integrand $\Phi\colon \Omega\times \R^d\to \R$ satisfying the p-growth condition
\begin{align} \label{s4_eq:p_subgrowth}
|\Phi(x,z)| \leq C(1+|z|)^p \quad \forall (x,z)\in \Omega\times \R^d, \tag{Gp}
\end{align}
we seek to represent limits of $\left( \int_\Omega \Phi(x,v_j(x))dx \right)_j$ as $j\to\infty$, possibly passing through suitable subsequences.

\begin{remark}
For each $j$, the map $\Phi$ acts on the graph of $v_j$, i.e. $\Phi(x,v_j(x)) = \Phi \circ (x,v_j(x))$. Therefore, we look for the limiting distribution of $(x,v_j(x))$ as $j\to \infty$, and more precisely the $\Phi$-moment of this limiting distribution.
\end{remark}

Morally speaking, since $(\Phi(\cdot,v_j))_j$ is bounded in $L^1(\Omega,\mu)$, $(\Phi(\cdot,v_j)d\mu)_j$ is bounded in $\M(\overline{\Omega})\eqsim C(\overline{\Omega})^*$, so by the abstract compactness principle \refp{Appendix:Banach-Alaoglu}, there exists a limit measure which depends on the integrand $\Phi$.

\subsubsection{Functional analytic setup}

Let $z\in \R^d \mapsto \hat{z} = \frac{z}{1+|z|}\in \B^d$ be a homeomorphism $\R^d\mapsto \B^d$. Define the class of functions of p-growth in the $z$ variable to be
\begin{align}
\G_p = \G_p(\Omega,\R^d) := \left\{ \Phi\in C(\Omega\times \R^d): \sup_{(x,z)} \frac{|\Phi(x,z)|}{(1+|z|)^p} < \infty \right\},
\end{align}
and for $\Phi\in \G_p$ put
\begin{align}
(T\Phi)(x,\hat{z}) := (1-|\hat{z}|)^p \Phi\left( x,\frac{\hat{z}}{1-|\hat{z}|} \right).
\end{align}
Then $T\colon \G_p \to \BC(\Omega\times \B^d)$ is an isometric isomorphism provided $\G_p$ is normed by $\|T\Phi\|_\infty$ and $\BC(\Omega\times \B^d)$ by $\|\cdot\|_\infty$. The inverse operator is
\begin{align}
(T^{-1}\Psi) (x,z) = (1+|z|)^p \Psi\left( x,\frac{z}{1+|z|} \right),
\end{align}
where $\Psi \in \BC(\Omega\times \B^d)$.

The dual operator $T^*\colon \BC(\Omega\times \B^d)^* \to \G_p^*$ is again an isometric isomorphism. We are interested in the limits of $\int_\Omega \Phi(x,v_j(x))dx$ for $\Phi \in \G_p$ and may define $\xi_{v_j} \in \G_p^*$ by
\begin{align} \label{s4:elementary_Y_legacy}
\xi_{v_j}(\Phi) := \int_\Omega \Phi(x,v_j(x))dx, \Phi \in \G_p.
\end{align}
Note
\begin{align}
\|\xi_{v_j} \| = \sup_{\| \Phi \|_{\G_p} } \xi_{v_j}(\Phi) = \int_\Omega (1+|v_j|)^pdx
\end{align}
so $(\xi_{v_j})$ is a bounded sequence in $\G_p^*$. But $\G_p^* \eqsim \BC(\Omega\times \B^d)^*$, and because $\BC$ (hence $\G_p$) is not separable we do not necessarily have sequential compactness. We must restrict the integrands $\Phi$ to a separable subspace of $\G_p$.

\subsection{Hausdorff compactification} \label{secHausCpt}

In this subsection, we present how to construct a compactification of the space $X=\Omega\times \R^d$ from a family of bounded and continuous functions $F\subset \BC(X)$. Roughly speaking, such compactification is a compact set $e_F X$ that contains $X$ as a dense subset and on which all $f\in F$ admit a continuous extension. The idea behind such construction is to look at the graph of each $f\in F$. Because of the boundedness assumption, each function has its image contained in a closed bounded interval of $\R$. Therefore, the graph is embedded into a closed subset of an (infinite-dimensional) hypercube, which is compact in the product topology by Tychonoff theorem, \refp{Appendix:Tychonoff}.

\subsubsection{Preliminaries on Hausdorff compactifications}

Most of the results will be stated without proof, which can be found in chapters 1 and 2 of \cite{chandler1976hausdorff} and in chapter 4 of \cite{folland1999real}.

We first define three classes of functions that are rich enough to determine the topological structure of their domains:
\begin{definition} \label{definitionSeparability}
Consider a family of functions $F\subset\BC(X)$, we say that $F$ separates points from closed sets if for each $C\subset X$ closed and $x\in X\setminus C$ there exists $f\in C(X)$ such that $f(x)\not\in \overline{f(C)}$.
\end{definition}

Next, we define what a compactification of a topological space is.
\begin{definition}
A compactification of $X$ is a compact Hausdorff space $\alpha X$ and an embedding $\alpha\colon X\to \alpha X$ (continuous and so that $\alpha^{-1}:\alpha X\to X$ exists and is continuous) such that $\alpha(X)$ is dense in $\alpha X$.
\end{definition}
It is useful to remark that because $\alpha$ is continuous, any function $f\in C(\alpha X)$ can be restricted to a continuous function on $X$. Indeed, $f \circ \alpha$ is the composition of a bounded continuous function with continuous function, and thus it belongs to $\BC(X)$. On the other hand, because $\alpha(X)$ is dense in $\alpha X$, each $f\in C(\alpha X)$ is uniquely recovered from $f\circ \alpha \in C(X)$.

Given a family $F\subset\BC(X)$ that separates points from closed sets, there is a canonical way of generating a compactification $\alpha X$ on which every $f\in F$ has a continuous extension.
\begin{theorem}\label{existenceCompact}
To each family $F\subset\BC(X)$ that separates points from closed sets, we associate a canonical embedding
\begin{align}
e_F\colon X\to \Pi_{f\in F} \big[ \inf f,\sup f \big],\ x\mapsto \{ f(x) \}_{f\in F}.
\end{align}
$e_F X:= \overline{e_F(X)}$ is a compactification of $X$
\end{theorem}
The above theorem is a direct implication of Tychonoff's theorem. When $F\subset\BC(X)$ is a family that separates points from closed sets, then $e_F\colon X\mapsto \Pi_{f\in F} \big[ \inf f,\sup f \big]$ is open and continuous, and so it is an embedding.

However, the map $e_F$ makes sense even if $F$ does not separate points from closed sets, and $e_F X$ is always a compact subset of $\Pi_{f\in F} \big[ \inf f,\sup f \big]$.

\begin{lemma}
Let $F\subset\BC(X)$ be a family that separates points from closed sets and $e_F X$ its induced compactification. Each $f\in F$ embeds into $C(e_F(X))$ in an obvious way and admits a unique extension $\overline{f}\in C(e_F X)$.
\end{lemma}
Every $y\in e_F X$ is an accumulation point of $e_F(X)$, so we can find a net $y_\lambda = \Pi_{f\in F} f(x_\lambda)$ such that $y_\lambda\to y$. A way of extending $f\in F$ is by setting
\begin{align}
\overline{f}\colon e_F X\to \R, y \left( = \lim_\lambda \Pi_{f\in F} f(x_\lambda) \right) \mapsto \overline{f}(y)= \lim_\lambda f(x_\lambda),
\end{align}
which does not depend on the choice of $x_\gamma$ as far as $\lim_\gamma f(x_\gamma) = \lim_\lambda f(x_\lambda)$ for each $f\in F$.

Suppose that we have given a family $F$ and its associated compactification $e_F X$, and we consider the compactification of $F\cup \{ f\}$, where $f\in\BC(X)$. We expect the latter compactification to be bigger than the former, i.e. to be a space where all the previous extensions can be further extended to continuous functions.

\begin{definition} \label{relation_compactifications}
Given two compactifications $\alpha X$ and $\gamma X$ of $X$, we say that $\alpha X\geq \gamma X$ if there exists a continuous function $f\colon \alpha X\to \gamma X$ such that $f\circ \alpha = \gamma$. Moreover, we write $\alpha X \eqsim \gamma X$ if $\alpha X \geq \gamma X$ and $\gamma X \geq \alpha X$, or equivalently if $f\colon \alpha X\to \gamma X$ is a homeomorphism.
\end{definition}

In the following paper, we will sometimes refer to a generic compactification $K$ without specifying the underlying family generating it. The reason why is stated by the following astonishing result.
\begin{theorem}
Given a compactification $\alpha X$ of $X$, there exists a family $F\subset\BC(X)$ that separates points from closed sets such that $e_F X\eqsim \alpha X$.
\end{theorem}

\subsubsection{Representation of compactifications}

Consider a family $F\subset\BC(X)$ that separates points from closed sets. According to the Hausdorff compactification theory (see above subsections), its induced compactification can be written as a subset of the hypercube $\Pi_{f\in F} \big[ \inf f,\sup f \big]$, where the sides of this cube are as many as the functions $f\in F$. Because $F$ can be uncountable, its compactification could be hard to deal with from an analytical point of view. We seek a better representation of such space.

The idea behind the following result is that if we know the limits of functions $f,g\in BC(\Omega,\R)$, we also know the limits of $f^n+g^m$, $n,m\in \N$.
\begin{definition}
Let $\F$ be a family of functions $f:X\to \R$. We call $\A(\F)$ the algebra generated by $\F$, i.e.
\begin{align}
\A(\F) = \{ f^n+g^m :f,g\in \F,\ n,m\in \N\}.
\end{align}
\end{definition}

Follows.
\begin{theorem}[Representation theorem] \label{subAlgebraRepresentation}
Let $F\subset\BC(X)$ be a closed sub-algebra that separates points from closed sets and let $F' \subset F$ be such that $\overline{\A(F')} = F$. Then $e_{F'}X$ is a compactification of $X$ and $e_{F'} X \eqsim e_F X$.
\end{theorem}

\begin{proof}
Let $e_{F'}X$ be the (formal) compactification of $F'$. Clearly $e_F X \geq e_{F'} X$. To prove the opposite inclusion, we must find a continuous function $T\colon e_{F'} X \to e_F X$ such that $T\circ e_{F'} = e_F$. Fix $y = \lim_\lambda \Pi_{f\in F'} f(x_\lambda) \in e_F X$. If $g \in \A(F')$, $\lim_\lambda g(x_\lambda)$ exists and coincides on all nets $x_\gamma$ such that $y = \lim_\gamma \Pi_{f\in F'} f(x_\gamma)$. Next, let $g\in \overline{\A(F')}$ and find a sequence $\{f_n\}_{n\in \N}\subset \A(F')$ such that $f_n\to g$ uniformly. Because $\|f_n\|_\infty$ is bounded, so is $\{ \lim_\lambda f_n (x_\lambda)\}_{n\in \N}$, and so we can extract a subsequence $\{f_{n_k}\}_{k\in \N}$ such that $\lim_k \lim_\lambda f_{n_k} (x_\lambda) = L\in \R$. Fix $\eps>0$ and $N\in \N$ such that $\|f_{n_N}-g\|_\infty < \frac{\eps}{3}$ and find $\tilde{\lambda}\in \Lambda$ such that $|f_{n_N}(x_\lambda)-\lim_\lambda f_{n_N}(x_\lambda)| < \frac{\eps}{3}$ for all $\lambda \geq \tilde{\lambda}$. Finally
\begin{align}
|g(x_\lambda)-L|\leq |g(x_\lambda)-f_{n_N}(x_\lambda)| + |f_{n_N}(x_\lambda)-\lim_\lambda f_{n_N}(x_\lambda)| + |\lim_\lambda f_{n_N}(x_\lambda) - L| < \eps
\end{align}
for all $\lambda \geq \tilde{\lambda}$, i.e. the net $g(x_\lambda)$ converges to $L\in \R$. In particular, by the uniqueness of $\lim_\lambda g(x_\lambda)$, we conclude that the original sequence $\{\lim_\lambda f_n(x_\lambda)\}_{n\in \N}$ converges to $L$, which also proves that the limit does not depend on the particular net $ x_\gamma$ as far as $\lim_\lambda f(x_\lambda) = \lim_\gamma f(x_\gamma)$ for all $f\in F'$. This shows that the map
\begin{align}
T\colon e_{F'} X\to e_F X, y = \lim_\lambda \Pi_{f\in F'} f(x_\lambda) \mapsto Ty = \lim_\lambda \Pi_{f\in F} f(x_\lambda).
\end{align}
is a well-defined isomorphism. Because its inverse is the projection map $\pi_{F'}|_{T(e_{F'}X)}$, which is continuous and open, then
\begin{align}
T\circ e_{F'}\colon x\mapsto \Pi_{f\in F'} f(x) \mapsto \Pi_{f\in F} f(x) = e_F(x)
\end{align}
is a homeomorphism. To prove that $e_{F'}X$ is a compactification of $X$, we notice that $F$ separates points, as so does $F'$. If $U\subset X$ is open, so is
\begin{align}
e_{F'}(U) = T^{-1}(e_F (U)),
\end{align}
and $e_{F'}$ is an injective continuous open map, thus it is an embedding onto its image.
\end{proof}

\subsection{Restriction on non-linearity}

For the sake of this work, it is important that the set of functions we work with is separable (has a countable dense set). Let $F \subset\BC(\Omega\times \B^d)$ be a closed separable algebra that separates points from closed sets and let $e_F\Omega\times \B^d$ be its compactification. Because $e_F\Omega\times \B^d$ is a compact Hausdorff space we have the following isometric isomorphism of its dual
\begin{align}
C(e_F \Omega\times \B^d)^* \eqsim \M(e_F \Omega\times \B^d).
\end{align}

\begin{lemma} \label{compatificationSeparable}
If $F\subset\BC(X)$ is separable, so is $C(e_F X)$.
\end{lemma}

\begin{proof}
Let $\{f_n\}_{n\in\N}$ be dense in $F$. By the Stone-Weierstrass theorem, the algebra generated by $\{1\}\cup \{\overline{f}_n\}_{n\in\N} \subset C(e_F X)$ is dense in $C(e_F X)$, and so $C(e_F X)$ is separable.
\end{proof}

Because $C(e_F\Omega\times \B^d)$ is separable we also have the abstract sequential compactness principle \refp{Appendix:Banach-Alaoglu} on its dual $\M(e_F\Omega\times \B^d)$. Let $T^{-1} F \subset \G_p$ the corresponding algebra (w.r.t. $\times^p$) on the set of continuous functions with p-growth. There is an isometric isomorphism
\begin{align}
T^{-1}F \overset{\tilde{T}}{\eqsim} C(e_F\Omega\times \B^d).
\end{align}
By the Riesz representation theorem, we can write its adjoint as
\begin{align}
\tilde{T}^*\colon \M(e_F\Omega\times \B^d) \to (T^{-1}F)^*, \nu \mapsto \left( \Phi \mapsto (\tilde{T}^*\nu, \Phi) = (\nu, \tilde{T}\Phi) = \int_{e_F\Omega\times \B^d} \tilde{T} \Phi d\nu. \right)
\end{align}

\begin{lemma} \label{productCompact}
Let $X,Z$ be completely regular Hausdorff spaces and let $F\subset\BC(X)$ and $G\subset\BC(Z)$ be closed sub-algebras that separate points from closed sets and contain the constant function. The following spaces are all isometrically isomorphic to each other
\begin{align}
C( e_{F\cup G} X\times Z ) \eqsim C(e_F X\times e_G Z) \eqsim \overline{\A(C(e_F X)\times C(e_G Z))} \eqsim \overline{\A( F\cup G)},
\end{align}
where each $f\in F$ and $g\in G$ is extended to a function on the product space by keeping constant the other variable. Moreover, $F\cup G\subset\BC(X\times Z)$ separates points from closed sets.
\end{lemma}

The proof of the above lemma is a straightforward application of the Stone-Weierstrass theorem. We underline that it's important to take families $F,G$ defined exclusively on each respective space, and the theorem is false if we instead add $f=f(x,y)$ that is not of the form above.

Morally speaking, what the previous lemma says is that on product spaces it is enough to work with the compactifications in each coordinate separately. Moreover, their dual elements (measures) can be tested against tensor products of functions that depend on each variable independently.

\subsection{Representation of Young measures}

Let $\Omega\subset \R^n$ be open and bounded and $G'\subset\BC(\B^d)$ and $F'\subset\BC(\Omega)$ be closed separable sub-algebras that separate points from closed sets. Let $T^{-1} F = A \subset \G_p$ the corresponding algebra, with respect to the $\times^p$ product, in the set of functions having p-growth. $F$ is isometrically isomorphic to $C(e_F \Omega\times \B^d)$. 

With abuse of notation, we are going to call $T$ the isomorphism between $A$ and $C(e_F \Omega\times \B^d)$. To each function $u\in L^p(\Omega,\R^d)$ we associate an \emph{elementary Young measure} $\xi_u\in A^*$ by setting
\begin{align}
\xi_u\colon A\to \R, \Phi \mapsto \int_\Omega \Phi(x,u(x))d\mu(x).
\end{align}
Next, consider a bounded sequence $\{u_n\}_{n\in \N}\subset L^p(\Omega,\R^d), \sup_n \|u_n\|_p \leq C$. As $\Phi \in \G_p$, the sequence of elementary Young measures is also bounded
\begin{align}
\|\xi_{u_n}\| = \sup_{\|\Phi\| \leq 1} \left| \int_\Omega \Phi(x,u_n(x))d\mu(x) \right| = \int_\Omega (1+|u_n|)^p d\mu \leq \mu(\Omega) + C^p.
\end{align}
Because of the isomorphism $A^* \eqsim C(e_F\Omega\times \B^d)^* \eqsim \M(e_F\Omega\times \B^d)$, there exists a subsequence, relabelled in the same way, and $\nu \in A^*$ such that $\xi_{u_n} \rightharpoonup^* \nu$ in $A^*$. Set
\begin{align}
L:= (T^*)^{-1} \nu \in \M(e_F\Omega\times \B^d).
\end{align}

We now study the measure $L$ to find a better representation for $\nu\in A^*$. Let $\psi \in C(e_F\Omega\times \B^d)$, we immediately notice that $L\in \M^+(e_F\Omega\times \B^d)$ as a consequence of the following equality
\begin{align}
\ll \nu, T^{-1} \psi \gg = \lim_n \int_\Omega (T^{-1}\psi)(x,u_n(x))d\mu(x).
\end{align}
Because the constant functions belong to $G'$, we can plug $T^{-1} \psi = \phi(x) (1+|z|)^p, \phi\in C(e_{F'} \Omega)$ into the previous equation and obtain the identity
\begin{align}
\int_{e_F \Omega\times \B^d} \phi(x) dL(x,\overline{z}) = \lim_n \int_\Omega \phi(x) (1+|u_n(x)|)^p d\mu(x) = \int_{e_{F'}} \phi(x)d\lambda(x).
\end{align}
By \refp{productCompact}, the projection
\begin{align}
\pi\colon e_F \Omega\times \B^d \to e_{F'} \Omega
\end{align}
is well-defined, and we can write $\tilde{\lambda} = \pi_\# L$. Note that hereby $\tilde{\lambda} \in \M^+(e_{F'}(\Omega))$.
%
%
Find the unique $\tilde{\lambda}$-measurable parametrized family $\{ \tilde{\nu}_x \}_{x\in e_{F'}\Omega}$ such that $\nu_x\in \M^+_1(e_{G'} \B^d)$ $\tilde{\lambda}$-almost every $x$ and
\begin{align}
\langle L,\Phi \rangle = \int_{e_{F'}\Omega } \langle \tilde{\nu}_x,\Phi(x,\cdot) \rangle d\tilde{\lambda}(x) \quad \forall\ \Phi\in C(e_F \Omega\times \B^d).
\end{align}
For any $\phi\in C_0(\Omega)$ take $\Phi = \phi (1-|\cdot|)^p$. We compute
\begin{align}
\int_{e_{F'}\Omega } \phi(x) \langle \tilde{\nu}_x, (1-|\cdot|)^p \rangle d\tilde{\lambda}(x) = & \int_{e_F \Omega\times \B^d} \phi(x) (1-|z|)^p dL(x,z) \\
= & \lim_n \int_\Omega (\phi \1_{\R^d})(x,u_n(x))d\mu(x) = \int_\Omega \phi(x) d\mu(x).
\end{align}
Because $\mu \in C_0(\Omega)^*$ we immediately conclude that
\begin{align}
\mu = \langle \tilde{\nu}_x, (1-|\cdot|)^p \rangle \tilde{\lambda} \mres \Omega,
\end{align}
where $\mu$ is extended on $e_{F'}\Omega$ by $\mu(E)\equiv \mu( e_{F'}^{-1}(E) ), E\subset e_{F'}\Omega$ Borel. Apply the Radon-Nikodym theorem and write
\begin{align}
\tilde{\lambda} = \frac{\tilde{\lambda}}{\mu} d\mu + \tilde{\lambda}^s.
\end{align}
From the previous identification, we get
\begin{align}
\begin{cases}
\langle \tilde{\nu}_x, (1-|\cdot|)^p \rangle \frac{\tilde{\lambda}}{\mu} = 1 \quad & \mu-\text{a.e.} \\
\langle \tilde{\nu}_x, (1-|\cdot|)^p \rangle = 0 & \tilde{\lambda}^s-\text{a.e.,}
\end{cases}
\end{align}
where the second condition implies that $\tilde{\nu}_x(e_{G'}(\B^d)) = 0$ $\tilde{\lambda}^s$-a.e., i.e. the measures are concentrated on the boundary $\partial e_{G'}(\B^d)$. On $e_{G'}(\B^d)$ we have $0 < (1-|z|)^p \leq 1$, whereas $|z|=1$ on $\partial e_{G'}(\B^d)$. In particular
\begin{align}
\begin{cases}
\frac{\tilde{\lambda}}{\mu} = \frac{1}{\langle \tilde{\nu}_x, (1-|\cdot|)^p \rangle} \geq 1  & \mu-\text{a.e.} \\
\tilde{\nu}_x(\partial e_{G'}(\B^d)) = 1 & \tilde{\lambda}^s-\text{a.e.}
\end{cases}
\end{align}
Now let $\phi \in \BC(\R^d)$ and define
\begin{align}
\langle \nu_x,\phi \rangle = & \frac{\tilde{\lambda}}{\mu}(x) \int_{e_{G'}(\B^d)} (1-|z|)^p \phi \left(\frac{z}{1-|z|} \right)  d\tilde{\nu}_x(z) \\
= & \frac{\tilde{\lambda}}{\mu}(x) \int_{\B^d} (1-|z|)^p \phi \left( \frac{z}{1-|z|} \right) d(e_{G'})_\# \tilde{\nu}_x(z)
\end{align}
In particular $\nu_x \in \M^+_1(\R^d)$ and $\{\nu_x\}_{x\in\Omega}$ is $\mu$-measurable. Let $\lambda = \tilde{\nu}_x(\partial e_{G'}(\B^d)) \tilde{\lambda}$. Then $\lambda\in \M^+(e_{F'}\Omega)$ and it decomposes into
\begin{align}
\lambda = & \tilde{\nu}_x(\partial e_{G'}(\B^d)) \frac{\tilde{\lambda}}{\mu} \mu + \tilde{\nu}_x(\partial e_{G'}(\B^d)) \tilde{\lambda}^s \\
= & \tilde{\nu}_x(\partial e_{G'}(\B^d)) \frac{\tilde{\lambda}}{\mu} \mu + \tilde{\lambda}^s.
\end{align}
For $\lambda$-almost every $x\in e_{F'} \Omega$ and for $\psi \in C(\partial e_{G'}\B^d)$ set
\begin{align}
\langle \nu_x^\infty, \psi \rangle = \frac{1}{\tilde{\nu}_x(\partial e_{G'}(\B^d))} \int_{\partial e_{G'}(\B^d)} \psi(z) d\tilde{\nu}_x(z),
\end{align}
hereby
\begin{align}
\nu_x^\infty \in \M^+_1(\partial e_{G'}(\B^d)).
\end{align}
For each $\Phi\in A$, its recession function is defined to be the restriction
\begin{align} \phi^\infty = \phi|_{e_{F'}\Omega \times \partial e_{G'}(\B^d)}.
\end{align}
Finally, we obtain the formula
\begin{align}
\langle L,T\Phi \rangle = & \int_{e_{F'}(\Omega)} \langle \tilde{\nu}_x, T\Phi(x,\cdot) \rangle d \tilde{\lambda} \\
= & \int_{e_{F'} \Omega} \int_{e_{G'}(\B^d)} T\Phi d\tilde{\nu}_x \left( \frac{\tilde{\lambda}}{\mu} d\mu + \overbrace{d\tilde{\lambda}^s}^{=0} \right) + \int_{e_{F'} \Omega}\fint_{\partial e_{G'}(\B^d)} T\Phi d\tilde{\nu}_x \ (\tilde{\nu}_x(\partial e_{G'}(\B^d)) d\tilde{\lambda}) \\
= & \int_\Omega \langle \nu_x,\Phi(x,\cdot) \rangle d\mu + \int_{e_{F'}\Omega } \langle \nu_x^\infty,\Phi^\infty(x,\cdot)\rangle d\lambda
\end{align}
and the representation of the Young Measure as the triple
\begin{align}
\nu = \big( \{\nu_x\}_{x\in\Omega}, \lambda, \{\nu_x^\infty\}_{x\in e_{F'} \Omega} \big).
\end{align}
where
\begin{align}
& \nu_x \in \M^+_1(\R^d) \text{ for } \mu \text{-almost every } x\in \Omega, \\
& \lambda \in \M^+(e_{F'}\Omega), \text{ and } \\
& \nu_x^\infty \in \M^+_1(\partial e_{G'}(\B^d)) \text{ for } \lambda \text{-almost every } x\in \overline{\Omega}.
\end{align}
We say that $u_n$ converges in the sense of Young measures to $\nu$, and write
\begin{align}
u_n\xrightarrow{Y^p(\mu,e_A )} \nu \quad \text{ or just } \quad u_n\xrightarrow{Y^p(\mu,A)} \nu,
\end{align}
where $\mu$ is the measure that "regulates and weights" oscillation and concentration of $u_n$, and $e_A$ is the compactification at infinity, generated by the family $A$.

From this point onwards the family $F'$ in $\Omega$ will always be the set of functions $C(\overline{\Omega})$.

\subsection{Properties of generalised Young Measures and connection to Young Measures on the sphere} \label{oscillatingFunction}

Here we show how the above construction generalises the more classical setting of Young Measures on the sphere, see \cite{reshetnyak1968weak} for the original idea behind their representation, and \cite{alibert1997non} for their modern implementation in the calculus of variations. We then study how the new representation for generalised Young Measures behaves geometrically, and its properties.

As a reminder, we state here the definition of integrands with a regular recession at infinity.
\begin{definition}
The set of integrands admitting a regular recession at infinity is
\begin{align}
\E_p(\Omega,\R^d) = \left\{\Phi\in C(\overline{\Omega}\times \R^d): \lim_{t\to\infty} \frac{\Phi(x,tz)}{t^p}\in \R \text{ locally uniformly in } (x,z) \in \overline{\Omega}\times \R^d \right\}.
\end{align}
\end{definition}

Because we intend to generalise the theory of Young Measures on functions with a regular recession, we need to extend the above class and at the same time preserve good topological properties of such a larger class. To do so, consider countably many functions $g_i \in \BC(\B^d)$ and their representations as integrands of $p$-growth $g_i(\frac{z}{1+|z|}) (1+|z|)^p$. We are interested in understanding how to represent, in a simple way, Young Measures relative to the compactification generated by $\E_p \cup \{g_i(\frac{z}{1+|z|}) (1+|z|)^p \}$. In the language of Hausdorff compactifications, set $G'=C(\overline{\B^d})\cup \{g_i,i\in\N \}$ and $F' = C(\overline{\Omega})$. Because $C(\overline{\B^d})\subset G'$, the closure of the algebra generated by either family is separable, separates points from closed sets and contains the constants. Call $F=G'\cup F'$, and without loss of generality, we can assume that $\|g_i\|\leq 1$ for all $i$.
\begin{lemma} \label{oscillatingMeasure}
$C( e_F \Omega\times \B^d)$ is isometrically isomorphic to $C(\overline{\Omega}\times \overline{graph(g_i)})$, where $(g_i)\colon\B^d\to [-1,1]^\N, z\mapsto (g_i(z))_{i\in\N}$ and the topology on the target space is the product topology. It is metrised by
\begin{align}
d(z,w) = |z-w| + \sum_i 2^{-i} |g_i(z)-g_i(w)|.
\end{align}
\end{lemma}
\begin{proof}
By \refp{productCompact} it is enough to prove that $C(e_{g_i,i\in\N}\B^d)\eqsim C(\overline{graph(g_i)})$. \refp{subAlgebraRepresentation} provides the isomorphism
\begin{align} \overline{\A(C(\overline{\B^d})\cup \{g_i,i\in\N\})} = \overline{\A(1,z^1,\ldots,z^d,g_i,i\in\N)}, \end{align}
and we conclude by noticing that $\overline{(1,z^1,\ldots,z^d,g_i,i\in\N)(\B^d)}$ is homeomorphic to $\overline{graph(g_i)}$. The topological equivalence between such metrics and the product topology is standard.
\end{proof}
When $g_i\in\C(\overline{\B^d})$, then $C(e_F \Omega\times \B^d) \eqsim C(\overline{\Omega}\times \overline{\B^d})$, and therefore we recover the usual sphere representation for the compactification induced by $\E_p$. This means the obvious, that we can add functions that have a regular recession and we still obtain the same space (up to homeomorphisms).

By definition, the compactification of $\B^d$ can be represented by the space $\Gamma$ of sequences $\{z_n\}_{n\in \N}\subset \B^d$ such that $z_n\to z\in \overline{\B^d}$ and $g_i(z_n)$ converges for all $i\in\N$, and two such sequences $\{z_n\}_{n\in \N}$ and $\{w_n\}_{n\in \N}$ are identified provided
\begin{align}
\lim_n |z_n-w_n| + \sum_i 2^{-i} |g_i(z_n)-g_i(w_n)|=0.
\end{align}

\begin{definition}
We call $e_{g_i,i\in\N}$ the compactification, and $\partial e_{g_i,i\in\N} = e_{g_i,i\in\N} \setminus graph(g_i,i\in \N)$, we can write the triple Young measure as
\begin{align}
\nu = \big( \{\nu_x\}_{x\in\Omega}, \lambda, \{\nu_x^\infty\}_{x\in \overline{\Omega}} \big),
\end{align}
where $\nu_x \in \M^+_1(\R^d)$ for $\mu$-almost every $x\in \Omega$, $\lambda \in \M^+(\overline{\Omega})$, and $\nu_x^\infty \in \M^+_1(\partial e_{g_i,i\in\N})$ for $\lambda$-almost every $x\in \overline{\Omega}$.
\end{definition}
Notice that $\partial e_{g_i,i\in\N}$ is an abuse of notation and refers to the boundary of the embedded space within the compactification.

So far we have constructed compactifications by "glueing" $g_i$s on top of the functions $z^1,\ldots,z^d$; that is to say on top of the unit ball. It is sometimes useful to iterate this argument, to stack another countable family $\{ f_i,i\in \N \}$ on top of the compactification $e_{g_i,i\in\N}$. This process gives the same compactification as if we were considering the two families at once, as the following lemma shows.
\begin{lemma}
Let $e_{g_i,i\in\N}$ be a compactification of $\B^d$ and $f_i\in BC(\B^d)$. Then
\begin{align}
e_{g_i,f_i,i\in\N} \eqsim \overline{ graph_{graph(g_i)} f_i }.
\end{align}
\end{lemma}

\begin{proof}
This is a trivial consequence of the fact that
\begin{align}
\{ (z^1,\ldots,z^d,g_i(z),f_i(z)), z\in \B^d \} = & \{ (z^1,\ldots,z^d,w,z): w=g_i(z), y=f_i(z), z\in \B^d \} \\
\text{ (extending $f_i$ to constant in the variable $w$) }= & \{ (z^1,\ldots,z^d,w,z): y = f_i(z,w), w=g_i(z), z\in \B^d \} \\
= & graph_{graph(g_i)}(f_i), z\in \B^d.
\end{align}
\end{proof}

A standard application of the disintegration lemma yields the following.

\begin{corollary} \label{disintegration_compactification_Y}
Consider a compactification $e_{g_i,f_i,i\in\N}$ and $\nu^\infty \in\M(\partial e_{g_i,f_i,i\in\N})$, then
\begin{align}
\nu^\infty = P_{(z_n)_{n\in \N}} d\tilde{\nu}^\infty
\end{align}
where $\tilde{\nu}^\infty \in \M(\partial e_{g_i})$, $(z_n)_n\in \partial e_{g_i}$, and $P_{(z_n)_n}$ is a probability measure defined on the space of subsequences $(z_{n_i})_i$ of $(z_n)_n$ so that $f_i (z_{n_i})$ converges for all $i\in\N$ (with sequences being equivalents if all the limits are).
\end{corollary}

For the case of oscillating functions $f_i$, we also write the compactification as $e_{f_i} X \equiv e_{f_i}$ and the convergence as
\begin{align}
v_j \xrightarrow{Y^p(\mu,f_i)} \nu.
\end{align}
When working with the sphere compactification, we will simply write
\begin{align}
v_j \xrightarrow{Y^p(\mu,\B^d)} \nu, \text{ or just } v_j \xrightarrow{Y^p(\mu)} \nu.
\end{align}
Also, because here we mainly consider the case $p=1$, we omit the superscript $p$ in $Y^p$ and write
\begin{align}
v_j \xrightarrow{Y(\mu,e_{f_i})} \nu.
\end{align}

We now study the relation of Young Measures with respect to different compactifications and different underlying measures $\mu \in \M^+(\Omega)$. Using Chacon \refp{Appendix:Chacon}, we can prove the following structure theorems.

%

\begin{lemma}
Let $v_j \xrightarrow{Y(\mu,e_{f_i,i\in\N})} (\nu_x,\lambda,\nu_x^\infty)$. Then for all $\psi\in C_0(\R^d)$, we have
\begin{align}
\psi(v_j) \rightharpoonup \langle \nu_x,\psi\rangle \text{ weakly in } L^1(\mu).
\end{align}
\end{lemma}
\begin{proof}
Because $\psi(v_j)\in L^\infty$ then is the sequence is equi-integrable and there is a subsequence that converges weakly in $L^1$ to $v$. Because $\psi^\infty = 0$, testing against $\phi\in \D(\Omega)$ we get
\begin{align}
\int_\Omega \phi \psi(u_j) d\mu \to \int_\Omega \phi v d\mu = \int_\Omega \phi \langle \nu_x,\psi\rangle d\mu.
\end{align}
\end{proof}

We can improve the above weak convergence result to show the following.

\begin{lemma} \label{oscillation_different_measures}
Let $u_j \xrightarrow{Y(\mu)} \big( \nu_x,0,N/A \big)$. For every $a\in L^1(\Omega,\mu)$ such that $a>0$ $\mu$-a.e. and for all $\psi \in C_0(\R^d)$ we have
\begin{align}
\psi \left( \frac{u_j}{a} \right)a \rightharpoonup \langle \nu_x, \psi \left( \frac{\cdot}{a(x)} \right)\rangle a(x) \text{ in } L^1(\Omega,\mu).
\end{align}
\end{lemma}

As expected, this implies that oscillations do not depend on the particular compactification chosen.

Before proving the above results we show the following uniform approximation result:
\begin{lemma}
Let $\mu\in \M^+(\Omega)$ and $a\in L^1(\Omega,\mu),a>0$ $\mu$-almost everywhere. There exists $a_n\in L^1(\Omega),0<a_n<a$, so that $a_n(x)\in \Q$ for all $x\in \Omega$ and
\begin{align}
\|a_n-a\|_\infty + \left\|  \frac{a}{a_n} - 1 \right\|_\infty \xrightarrow{n\to \infty} 0.
\end{align}
\end{lemma}

\begin{proof}
Let
\begin{align}
a_n = \sum_{k\in \N,k\geq 1} \chi_{ a^{-1}\big( [\frac{k}{n},\frac{k+1}{n}) \big) } \frac{k}{n},
\end{align}
where $\N$ is the set of strictly positive integers. Because $a_n\leq a$ then $a_n\in L^1$ and it also assumes countably many values at a time. Also $|a_n(x)-a(x)| \leq \frac{1}{n}$ so it converges uniformly to $a$. Furthermore
\begin{align}
1 = \frac{\frac{k}{n}}{\frac{k}{n}} \leq \frac{a(x)}{a_n(x)} \leq \frac{\frac{k+1}{n}}{\frac{k}{n}} = \frac{k+1}{k}
\end{align}
and so $\frac{a}{a_n}$ converges uniformly to 1.
\end{proof}

Now we can prove \refp{oscillation_different_measures}
\begin{proof}
Using the previous approximation, we write, for 1-Lipschitz $\psi:\R^d\to \R$,
\begin{align}
\int_\Omega \psi \left( \frac{u_j}{a} \right)a = & \int_\Omega \overbrace{ \psi \left( \frac{u_j}{a} \right)a -\psi\left(\frac{u_j}{a_n} \right)a }^{=I} + \overbrace{ \psi \left( \frac{u_j}{a_n} \right)a - \psi \left( \frac{u_j}{a_n} \right)a_n }^{=II} + \psi \left( \frac{u_j}{a_n} \right)a_n.
\end{align}
The first two terms are bounded by
\begin{align}
|I| \leq & \int_\Omega a \left| \frac{u_j}{a} - \frac{u_j}{a_n} \right| = \int_\Omega |u_j| \left| 1-\frac{a}{a_n} \right| \leq \sup_j \|u_j\| \left\|1-\frac{a}{a_n} \right\|_\infty \\
|II| \leq & \int_\Omega \left( 1+\frac{|u_j|}{a_n} \right) |a-a_n| \leq (1+\sup_j |u_j|) \left\| 1 -\frac{a}{a_n} \right\|_\infty,
\end{align}
which goes to $0$ as $n\to \infty$ uniformly in $j$. As for the third term, calling $E_k = a^{-1} \big( [ \frac{k}{n},\frac{k+1}{n}) \big)$, we can use dominated convergence theorem to pass to the limit
\begin{align}
\lim_j \int_\Omega \psi \left( \frac{u_j}{a_n} \right) a_n = \lim_j \sum_k \int_{E_k} \psi \left( \frac{u_j}{\frac{k}{n}} \right) \frac{k}{n} = \sum_k \int_{E_k} \langle \nu_x,\psi \left( \frac{\cdot}{\frac{k}{n}} \right) \rangle \frac{k}{n} = \int_\Omega \langle \nu_x,\psi \left( \frac{\cdot}{a_n} \right) \rangle a_n.
\end{align}
Another application of the dominated convergence theorem will let us conclude the result.
\end{proof}

We can finally conclude with a structure theorem regarding concentration.

\begin{proposition} \label{structure_comparison_Y}
Consider two separable algebras (that separate points from closed sets) $A$ and $B$ of $\G_1$ and let $a\in L^1(\Omega,\mu), a>0$ $\mu$-a.e. Let $v_j\in L^1(\Omega,\mu)$ be a sequence so that
\begin{align}
v_j \xrightarrow{Y(\mu,A)} \big( \nu_x,\lambda_\nu,\nu_x^\infty \big) \quad \text{ and } \quad \frac{v_j}{a} \xrightarrow{Y(a\, d\mu,B)} \big( \eta_x,\lambda_\eta,\eta_x^\infty \big).
\end{align}
Then $\nu_x = \left( \frac{\cdot}{a(x)} \right)_\# \eta_x$, $\lambda_\nu = \lambda_\eta = \lambda$. Moreover, decomposing
\begin{align}
\nu_x^\infty = P^\nu_{(z_n)_n} d\tilde{\nu}_x^\infty \quad \text{ and } \quad \eta_x^\infty = P^\eta_{(z_n)_n} d\tilde{\eta}_x^\infty,
\end{align}
where $\tilde{\nu}_x^\infty$ and $\tilde{\eta}_x^\infty$ are the projections on the sphere according to \refp{disintegration_compactification_Y}, then $\tilde{\nu}_x^\infty=\tilde{\eta}_x^\infty$ $\lambda$-a.e. with
\begin{align}
v_j \xrightarrow{Y(\mu,\B^d)} \big( \nu_x,\lambda,\tilde{\nu}_x^\infty \big).
\end{align}
\end{proposition}

\begin{proof}
For all $\psi \in C_0(\R^d)$ we have that $\psi(v_j)$ is equi-integrable so that
\begin{align}
\psi(v_j) \rightharpoonup \langle \eta_x,\psi\rangle \text{ in } L^1(\mu)
\end{align}
and
\begin{align}
\psi \left( \frac{v_j}{a} \right) \rightharpoonup \langle \nu_x,\psi\rangle \text{ in } L^1(ad\mu).
\end{align}
Next, identify $v_j$ with its subsequence and find $E_k$ so that $v_j \rightharpoonup v$ in $L^1(E_k,\mu)$ for all $k$. By inner approximation, we can assume that all such $E_k's$ are compact. Consider now the sequence $v_j \chi_{E_k}$. Then
\begin{align}
v_j \chi_{E_k} \xrightarrow{Y(\mu,A)} \big( \eta_x \chi_{E_k} + \delta_0 \chi_{E_k^c},0,N/A \big).
\end{align}
By \refp{oscillation_different_measures} we then have, for $\phi \in C_0(\Omega)$ and $\psi \in C_0(\R^d)$, because $\Omega\setminus E_k$ is open,
\begin{align}
\int_\Omega \phi \langle \nu_x,\psi\rangle a(x)d\mu = & \lim_j \int_\Omega \phi \psi \left( \frac{v_j}{a} \right) a d\mu = \lim_j \int_{\Omega\setminus E_k} \phi \psi \left( \frac{v_j}{a} \right) a d\mu + \int_{E_k} \phi \psi \left( \frac{v_j}{a} \right) a d\mu \\
= & \int_{\Omega\setminus E_k} \phi \langle \nu_x,\psi \rangle ad\mu + \int_{E_k} \phi \langle \eta_x, \psi \left( \frac{\cdot}{a} \right)\rangle ad\mu.
\end{align}

Next, let $\phi\in C(\overline{\Omega})$, then
\begin{align}
&& \lim_j \int_\Omega \phi |v_j| d\mu = & \int_\Omega \langle \nu_x,|\cdot|\rangle \phi d\mu + \int_{\overline{\Omega}} \phi d\lambda_\nu \\
&& = \lim_j \int_\Omega \phi \left| \frac{v_j}{a} \right| a d\mu = & \int_\Omega \langle \eta_x,\frac{|\cdot|}{a(x)}\rangle \phi a(x) d\mu + \int_{\overline{\Omega}} \phi d\lambda_\eta.
\end{align}
Using the previous part we conclude that $\lambda_\nu = \lambda_\eta = \lambda$.

Finally, let $f\in C(\partial \B^d)$ and extending by 1-homogeneity we obtain that
\begin{align}
\lim_j \int_\Omega \phi f(v_j)d\mu = & \int_\Omega \phi \langle \nu_x,f\rangle d\mu + \int_{\overline{\Omega}} \phi \langle \nu_x^\infty,f^\infty \rangle d\lambda = \int_\Omega \phi \langle \nu_x,f\rangle d\mu + \int_{\overline{\Omega}} \phi \int\int f^\infty dP^\nu_{(z_n)} d\tilde{\nu}_x^\infty d\lambda_\nu \\
= & \int_\Omega \phi \langle \nu_x,f\rangle d\mu + \int_{\overline{\Omega}} \phi \int f^\infty d\tilde{\nu}_x^\infty d\lambda_\nu \\
= & \int_\Omega \phi \langle \nu_x,f \left( \frac{\cdot}{a(x)} \right) \rangle a(x) d\mu + \int_{\overline{\Omega}} \phi \langle \eta_x^\infty,f^\infty \rangle d\lambda \\
= & \int_\Omega \phi \langle \nu_x,f\rangle d\mu + \int_{\overline{\Omega}} \phi \int f^\infty dP^\eta_{(z_n)} d\tilde{\eta}_x^\infty d\lambda_\eta \\
= & \int_\Omega \phi \langle \nu_x,f\rangle d\mu + \int_{\overline{\Omega}} \phi \int f^\infty d\tilde{\eta}_x^\infty d\lambda_\eta.
\end{align}
\end{proof}

Notice that the previous identification with the concentration angle measure fails if we only consider $\Gamma = A\cap B$ which does not necessarily generate the sphere compactification. This is so because sequences $(u_j)_j$ can concentrate around values of $a$ that are measure-discontinuous. However, equality holds true if $a=1$.

\begin{lemma}
Following the assumptions of \refp{structure_comparison_Y}, if $a=1$, $\Gamma = A\cap B$ and writing
\begin{align}
\nu_x^\infty = P^\nu_{(z_n)_n} d(\gamma^\nu)_x^\infty \quad \text{ and } \quad \eta_x^\infty = P^\eta_{(z_n)_n} d(\gamma^\eta)_x^\infty,
\end{align}
where $\gamma^\eta$ and $\gamma^\nu$ are the projections onto the compactification generated by $\Gamma$, then
\begin{align}
(\gamma^\nu)^\infty_x = (\gamma^\eta)_x^\infty \quad \lambda \text{-a.e.}
\end{align}
\end{lemma}

\begin{proof}
This is proven similarly at the end of \refp{structure_comparison_Y} and testing against functions belonging in $\overline{\A(\Gamma)}$ and using the decomposition of angle Young Measures.
\end{proof}

Next, we show that the lack of concentration is equivalent to the equi-integrability of the generating sequence.

\begin{theorem}
Let $(v_j)_j\in L^1(\Omega,\mu,\R^d)$ be so that
\begin{align}
v_j \xrightarrow{Y(\mu,e_{f_i,i\in \N})} \big( \nu_x,\lambda,\nu_x^\infty \big).
\end{align}
Then the sequence $(v_j)_{j\in \N}$ is equi-integrable if and only if $\lambda=0$.

Moreover $v_j\to v$ strongly in $L^1(\Omega,\mu,\R^d)$ if and only if $\lambda=0$ and $\nu_x=\delta_{v(x)}$ for $\mu$-a.e. $x\in \Omega$.
\end{theorem}

\begin{proof}
Because $\lambda$ does not depend on the compactification (see \refp{structure_comparison_Y}), we can apply the same theorem from the sphere compactification, \citep{rindler2018calculus}{347}, lemma 12.14 and \citep{rindler2018calculus}{348}, corollary 12.15, to conclude. 
\end{proof}

Before stating the next two structure results, we prove that $T^{-1}$ is a bounded operator from $Lip(e_{f_i,i\in\N})$ to $Lip(\R^d)$, provided the compactification is generated by Lipschitz functions. In this case, by $Lip(\R^d)$ we mean the weighted norm
\begin{align}
\|f\|_{Lip(\R^d)} := \| Tf \|_\infty + \sup_{x\neq y} \frac{|f(x)-f(y)|}{|x-y|} = \left\| \frac{f}{1+|\cdot|} \right\|_\infty +\sup_{x\neq y} \frac{|f(x)-f(y)|}{|x-y|}.
\end{align}
As for the compactification, the metric is always intended as in \refp{oscillatingMeasure}.

\begin{lemma} \label{Lipschitz_pullback}
Let $e_{f_i,i\in\N}$ be a separable compactification metrised by the usual metric, where $f_i\in Lip(\R^d)$ are normalised so that $\|f\|_{Lip(\R^d)}\leq 1$. Then
\begin{align}
\sup_{x\neq y} \frac{|g(x)-g(y)|}{|x-y|} \leq 5 Lip(Tg,e_{f_i,i\in\N})
\end{align}
for all maps $g\colon \R^d \to \R$.
\end{lemma}

\begin{proof}
Without loss of generality assume that $\|Tg\|_{Lip}\leq 1$, i.e. for all $|x|,|y|<1$,
\begin{align}
\left| g \left( \frac{x}{1-|x|} \right) (1-|x|) - g \left( \frac{y}{1-|y|} \right) (1-|y|) \right| \leq |x-y| + \sum_i 2^{-i} |Tf_i(x)-Tf_i(y)|.
\end{align}
Then
\begin{align}
|g(x)-g(y)| = & \left| \frac{g(x)}{1+|x|} (1+|x|) - \frac{g(x)}{1+|x|} (1+|y|) + \frac{g(x)}{1+|x|} (1+|y|) - \frac{g(y)}{1+|y|} (1+|y|) \right| \\
= & \frac{|g(x)|}{1+|x|} \Big|1+|x|-1-|y|\Big| + (1+|y|) \left| \frac{g(x)}{1+|x|} - \frac{g(y)}{1+|y|} \right| \\
\leq & |x-y| + (1+|y|) \left( \left| \frac{x}{1+|x|} - \frac{y}{1+|y|} \right| + \sum_i 2^{-i} \left|Tf_i(\frac{x}{1+|x|})-Tf_i(\frac{y}{1+|y|}) \right| \right).
\end{align}
For all $i$ we have that
\begin{align}
\left| Tf_i \left( \frac{x}{1+|x|} \right) - Tf_i \left( \frac{y}{1+|y|} \right) \right| = \left| \frac{f_i(x)}{1+|x|} - \frac{f_i(y)}{1+|y|} \right|,
\end{align}
and therefore after multiplying by $1+|y|$ we obtain
\begin{align}
\left| \frac{f_i(x)}{1+|x|}\Big(1+|x| + (|y|-|x|)\Big) - f_i(y) \right| = & \left| f_i(x) - f_i(y) + \frac{f_i(x)}{1+|x|}(|y|-|x|) \right| \\
\leq & |f_i(x)-f_i(y)| + \frac{|f_i(x)|}{1+|x|} |x-y| \leq 2|x-y|.
\end{align}
\end{proof}

We now show that the above lemma allows us to test Young Measures on Lipschitz compactifications against Lipschitz functions of $\R^d$.

\begin{definition}[Kantorovich semi-norm]
Let $X$ be a metric space and $\mu\in \M(X)$, then the (formal) Kantorovich norm of $\mu$ is
\begin{align}
\|\mu\|_K = \sup_{\|\phi\|_{Lip}\leq 1} \int_X \phi d\mu.
\end{align}
\end{definition}
The above formula induces a pseudo-distance between measures by setting $d(\mu,\eta)_K = \|\mu-\eta\|_K$. It turns out that this is indeed a metric on the positive cone of non-negative Measures, \refp{Appendix:Kantorovich}.In particular, by taking $\Psi\in Lip(e_{f_i,i\in\N})$ and the pull-back $T^{-1}$ we deduce the following.
\begin{lemma}
Let $e_{f_i,i\in\N}$ be a separable compactification. Then every $\nu \in Y(e_{f_i,i\in\N},\mu)$ is defined by testing it against Lipschitz functions of the form
\begin{align}
\phi \otimes \psi, \ \text{where } \|\phi\|_{Lip(\Omega)}\leq 1,\ \|\psi\|_{Lip(\R^d)}\leq 1.
\end{align}
\end{lemma}

%
%

For this reason, we remind once again of the norm we will be using on the space $e_{f_i,i\in\N}$ throughout this thesis.

\begin{definition}
We say that $e_{f_i,i\in\N}$ is a Lipschitz compactification if each $f_i$ is Lipschitz continuous, and renormalised so that $Lip(Tf)\leq 1$. The norm on $Lip(e_{f_i,i\in\N})$ will always be
\begin{align}
\|g\|_{Lip(e_{f_i,i\in\N})} := \sup_{x\in e_{f_i,i\in\N}} |g(x)| + \sup_{x\neq y} \frac{|g(x)-g(y)|}{d_{ e_{f_i,i\in\N}}(x,y) }
\end{align}
where
\begin{align}
d_{e_{f_i,i\in\N}}(x,y) = |x-y| + \sum_i 2^{-i} |Tf_i(x)-Tf_i(y)|.
\end{align}
\end{definition}

To conclude this subsection, we state decomposition results for Young Measures regarding oscillation and concentration. Originally proven in the context of the sphere compactification, \citep{kristensen2019oscillation}{29}, we here extend them to general compactifications.

\begin{lemma}
Let $v_j\in L^1(\Omega,\mu)$ so that
\begin{align}
v_j \xrightarrow{Y(e_{f_i,i\in\N},\mu)} \big( \nu_x,\lambda,\nu_x^\infty \big).
\end{align}
We can write $v_j = o_j + c_j$, where $o_j\in L^1(\Omega,\mu)$ is equi-integrable,
\begin{align}
o_j \xrightarrow{Y(e_{f_i,i\in\N},\mu)} \big( \nu_x,0,N/A \big)
\end{align}
and $c_j \in L^1(\Omega,\mu)$ so that
\begin{align}
c_j \xrightarrow{Y(e_{f_i,i\in\N},\mu)} \big( \delta_0,\lambda,\nu_x^\infty \big).
\end{align}
The converse is also true, for each such sequence $o_j,c_j$ as above, their sum converges to the former Young Measure. 
\end{lemma}

\begin{proof}
A standard diagonal argument gives us $k_j \uparrow \infty$ so that $o_j = v_j \chi_{|v_j|}\leq k_j$ is equi-integrable and generates $o_j \xrightarrow{Y(e_{f_i,i\in\N},\mu)} \big( \nu_x,0,N/A \big)$. Then, letting $c_j=v_j-o_j$, for $\eta\in C(\overline{\Omega})$, $T\psi\in C(e_{f_i,i\in\N})$ we have
\begin{align}
\int_\Omega \eta (\psi(c_j)-\psi(v_j)) = & \int_{|v_j|\leq k_j} \eta(\psi(0)-\psi(o_j)) + \int_{|v_j|>k_j} \eta(\psi(v_j)-\psi(v_j)) \\
= & \int_\Omega \eta(\psi(0)-\psi(o_j)) \to \int_\Omega \eta(\psi(0)-\langle \nu_x,\psi\rangle).
\end{align}
Writing $\psi(c_j) = \Big(\psi(c_j)-\psi(v_j)\Big) + \psi(v_j)$ and letting $j\to \infty$ we conclude.
\end{proof}

We remark here that, when considering certain subsets of $Y$ (for example Young Measures generated by gradients, see next section), $o_j$ and $c_j$ might generate different types of Young Measures.

The next lemma is an extension of the previous result.

\begin{lemma} \label{join_sequence_Y}
Let $v_j\in L^1(\mu)$ and $w_j\in L^1(\mu)$ generate
\begin{align}
v_j \xrightarrow{Y(\mu,e_{f_i,i\in\N})} \big( \delta_{v(x)},\lambda_\eta,\eta_x^\infty \big) \quad \text{ and } \quad w_j \xrightarrow{Y(\mu,e_{f_i,i\in\N})} \big( \nu_x,\lambda_\nu,\nu_x^\infty \big)
\end{align}
with $\lambda_\eta \perp \lambda_\nu$, for some $v\in L^1(\mu)$. Then the sum of the sequence generates
\begin{align}
v_j+w_j \xrightarrow{Y(\mu,e_{f_i,i\in\N})} \big( \delta_{v(x)} \ast \nu_x,\lambda_\nu + \lambda_\eta,k_x^\infty \big),
\end{align}
where
\begin{align}
k_x^\infty = \begin{cases}
\nu_x^\infty & \lambda_\nu \text{-a.e.} \\
\eta_x^\infty & \lambda_\eta \text{-a.e.}
\end{cases}
\end{align}
\end{lemma}

\begin{proof}
Write $w_j = o_j + c_j$ as in the previous lemma and put $b_j = v_j-v$. We claim that
\begin{align}
b_j+c_j \xrightarrow{Y(\mu,e_{f_i,i\in\N})} \big( \delta_0,\lambda_\nu + \lambda_\eta,k_x^\infty \big).
\end{align}
No oscillation is a consequence of the fact that $b_j+c_j\to 0$ in $\mu$-measure. Next, let $\phi\in C(\overline{\Omega}),\|\phi\|_{Lip}\leq 1$ and $\Psi\in e_{f_i,i\in\N},\|T\Phi\|_{Lip}\leq 1$ with $\Psi(0)=0$. Let $E_\nu$ and $E_\eta$ be sets where $\lambda_\nu$ and $\lambda_\eta$ are concentrated, respectively. For $\eps>0$ find $C_\nu\subset E_\nu, C_\eta\subset E_\eta$ compact sets and $O_\nu \supset C_\nu,O_\eta\supset C_\eta$ open sets such that
\begin{align}
\lambda_\eta(\overline{\Omega}\setminus C_\eta) + \lambda_\nu(O_\eta) + \lambda_\nu(\overline{\Omega}\setminus C_\nu) + \lambda_\eta(O_\nu) <\eps.
\end{align}
Consider a function $\rho\in C(\R^n)$ with $\chi_{C_\eta} \leq \rho \leq \chi_{O_\eta}$. We write
\begin{align}
\int_\Omega \phi \big( \Psi(b_j+c_j)- \Psi(b_j)-\Psi(c_j) \big) (\overbrace{\rho}^{=I} + \overbrace{1-\rho}^{II})d\mu.
\end{align}
We estimate the first guy by
\begin{align}
\limsup_j |I| \leq \limsup_j 2\int_\Omega \rho |b_j| = 2\int_{\overline{\Omega}} \rho d\lambda_\nu \leq 2\lambda_\nu(\overline{\Omega}\cap O_\eta) \leq 2\eps.
\end{align}
Similarly for the second term
\begin{align}
\limsup_j |II| \leq \limsup_j 2 \int_\Omega (1-\rho) |c_j| \leq 2 \lambda_\eta(\overline{\Omega}\setminus C_\eta) \leq 2\eps.
\end{align}
But the first term converges to $0$ and therefore we conclude for the representation of Young Measures.
\end{proof}

\subsection{Terminology}

We dedicate this part to clarifying the terminology of Young Measures adopted throughout this paper.

\begin{definition}
Given a separable algebra $A$ of $\G_p$ that separates points from closed sets, a \textit{p-Young measure} is a triple
\begin{align}
\nu = \big( (\nu_x)_{x\in \Omega},\lambda,(\nu_x^\infty)_{x\in \overline{\Omega}} \big)
\end{align}
where
\begin{enumerate}
\item $(\nu_x)_{x\in \Omega}$ is $\mu$-measurable and $\nu_x\in \M^+_1(\R^d)$ $\mu$-a.e. $x\in \Omega$.

We call it the \textit{oscillation Young measure};

\item $\lambda\in \M^+(\overline{\Omega})$.

We call it the \textit{concentration measure};

\item $(\nu_x^\infty)_{x\in \overline{\Omega}}$ is $\lambda$-measurable and $\nu_x^\infty \in \M^+_1(\partial e_A )$ $\lambda$-a.e. $x\in \overline{\Omega}$.  
We call it the \textit{concentration angle Young measure};

\item the moment condition
\begin{align}
\int_\Omega \int_{\R^d} |z|^p d\nu_x(z) < \infty
\end{align}
must hold.
\end{enumerate}
The collection of all such triples is denoted by $Y^p = Y^p(\Omega,\mu,e_A)$.
\end{definition}
Where obvious from the context, we will not specify the domain $\Omega$, the measure $\mu$, the family $A$ or the target space $\R^d$. Also, when $\lambda=0$, i.e. when there is no concentration, there is no point in specifying the compactification we are working with.

From every triple $\nu = \big( \nu_x,\lambda,\nu_x^\infty \big)$ one can construct the measure $L\in C(e_F \Omega\times \R^d)$ and vice-versa.
\begin{lemma}
The following equality holds:
\begin{align}
Y^p = T^*\left\{ L\in \M^+(e_F \Omega\times \B^d)^+: \int_{e_F \Omega\times \B^d} \phi(x)(1-|\hat{z}|)^p dL = \int_\Omega \phi(x)d\mu(x) \ \forall \phi\in C(e_G\Omega) \right\}.
\end{align}
In particular $Y^p$ is a weak* closed and convex subset of $A^*$.
\end{lemma}

\begin{proof}
Let $L\in \M^+(e_F \Omega\times \R^d)^+$ as above. It was already shown that $T^*L \in Y^p$. On the other side, if $\nu = \big( \nu_x,\lambda,\nu_x^\infty\big)$, we let
\begin{align}
L = \nu_x d\mu + \nu_x^\infty d\lambda.
\end{align}
Testing against a test function $\phi=\phi(x)$ that only depends on $x$,
\begin{align}
\int_{e_F\Omega\times\B^d} T\phi dL = \int_{e_F \Omega\times \B^d} \phi(x)(1-|\hat{z}|)^p dL = \int_\Omega \phi d\mu.
\end{align}
Conclude by noticing that the above characterisation amounts to
\begin{align}
Y^p = T^* \left( \bigcup_{\phi \in C(e_F\Omega)} \left\{ L\in \M^+(e_F \Omega\times\B^d)^+ : \int_{e_F \Omega\times \B^d} \phi(x)(1-\hat{z})^p dL = \int_\Omega \phi d\mu \right\} \right).
\end{align}
\end{proof}

\begin{remark}
With a straightforward adaptation of a classical argument for the sphere compactification, one can prove that given any Young Measure of the above form $\nu \in Y(\Omega,\mu,e_{f_i,i\in \N})$ with $\supp(\mu\mres \overline{\Omega}) = \overline{\Omega}$ and $\mu$ non-atomic, then there is a sequence of smooth functions $u_j\in \D(\Omega,\R^d)$ such that
\begin{align}
u_j \xrightarrow{Y(\Omega,\mu,e_{f_i,i\in\N})} \nu.
\end{align}
We do not transcribe the proof here because it won't be used at any point in this work.
\end{remark}

We now give a formal definition of what elementary Young Measures are, as a way to embed functions and measures.

\begin{definition}
Let $\mu\in \M^+(\Omega)$ and $v\in L^p(\Omega,\mu,\R^d), 1\leq p\leq \infty$. The corresponding \textit{elementary p-Young measure} is
\begin{align}
\xi_v := \big( (\delta_{v(x)})_{x\in \Omega},0,\text{N/A} \big) \in Y(\mu).
\end{align}
When $p=1$, we extend the definition to $l \in \M(\Omega,\R^d)$ by setting, for $l = \frac{l}{\mu} d\mu + l^{s,\mu}$,
\begin{align}
\xi_l := \left( \left( \delta_{\frac{l(x)}{\mu(x)}} \right)_{x\in \Omega},|l^{s,\mu}|, \left( \delta_\frac{l^{s,\mu}}{|l^{s,\mu}|} \right)_{x\in \Omega} \right) \in Y(\mu,\partial \B^d)
\end{align}
\end{definition}
Note that for $\Phi\in \E_p$ we have
\begin{align}
\ll \xi_v,\Phi \gg = & \int_\Omega \Phi(x,v(x))d\nu(x) \quad \text{ and } \\
\ll \xi_l,\Phi \gg = & \int_\Omega \Phi\left( x,\frac{l}{\mu}(x) \right) d\mu(x) + \int_\Omega \Phi^\infty\left( x,\frac{l^{s,\mu}}{|l^{s,\mu}|}(x) \right) d|l^{s,\mu}|(x).
\end{align}

In general, there is no clear way of defining elementary Young Measures on compactifications that are larger than the sphere. We will see later, however, that this can be done in very specific cases when we have more structure on $A$ and more information on the measure $l\in \M(\Omega,\R^d)$ we are trying to embed.

\begin{definition}[Barycentre of a p-Young measure]
Let $\nu = \big( (\nu_x)_{x\in \Omega}, \lambda, (\nu_x^\infty)_{x\in \overline{\Omega}} \big) \in Y^p(\mu,A)$, where $A$ is so that $e_A \geq \overline{\B^d}$ (in the sense of compactifications, see \refp{relation_compactifications}). We call its barycentre
\begin{align}
\overline{\nu} = \begin{cases}
\overline{\nu_x} \quad & 1<p<\infty \\
\overline{\nu_x} \mu + \overline{\nu^\infty_x} \lambda & p=1,
\end{cases}
\end{align}
which is the following quantity
\begin{align}
\overline{\nu_x} = & \int_{\R^d} z d\nu_x(z) \\
\overline{\nu^\infty_x} = & \int_{\partial e_A} zd\nu_x^\infty.
\end{align}
\end{definition}
In the above definition, "$\geq$" is the ordering over the set of Hausdorff compactifications of a topological space (see the subsection on Hausdorff compactifications). Moreover, the barycentre does not depend on the compactification, as far as $e_A \geq \overline{\B^d}$. Indeed $z = [z^j]_{j=1,\ldots,d}$ extended to $e_A$ coordinate-wise, and so
\begin{align}
\int_{\partial e_A} zd\nu_x^\infty = \int_{\partial \B^d} \int_{\{ (w_n)_n \}} zd P_z( (w_n)_n ) d\pi_{\partial \B^d} \nu_x^\infty = \int_{\partial \B^d} z d\pi_{\partial \B^d} \nu_x^\infty,
\end{align}
as the coordinate map $z\mapsto z^j$ is constant on sequences $(w_n)_n$ that converge to the same value $w\in \partial\B^d$.

Notice that $x\mapsto \overline{\nu_x}$ is $\mu$-measurable and $x\mapsto \overline{\nu^\infty_x}$ is $\lambda$-measurable. In particular,
\begin{align}
\overline{\nu} \in L^p(\Omega,\mu,\R^d) \text{ for } 1<p<\infty \quad \text{ and } \quad \overline{\nu}\in \M(\Omega,\R^d) \text{ for } p=1.
\end{align}
It is easy to see that when $1<p<\infty$, $v_j\rightharpoonup v\in L^p(\Omega,\mu,\R^d)$ we have $v=\overline{\nu_v}$, and when $p=1$, $\rho_j \rightharpoonup^* \rho \in C_0(\Omega,\R^d)^*$, then $\rho = \overline{\xi_\rho}$.

\subsection{Stronger notions of convergence}

To conclude the discussion about generalised Young Measures, we mention some stronger notions of convergence such as strict convergence and $\mu$-strict convergence. These modes of convergence explain why we chose such canonical embedding for measures in the previous paragraph. Moreover, we give a simple example of why such canonical embedding has no meaning when the compactification is larger than the sphere one.

\begin{definition}
Let $\eta_j,\eta\in \M(\Omega,\R^d)$, we say that $\eta_j \xrightarrow{s} \eta$ ($\eta_j$ converges strictly to $\eta$) if $\eta_j\to \eta$ weakly* in $C_0(\Omega,\R^d)^*$ and $|\eta_j|(\Omega)\to |\eta|(\Omega)$.
\end{definition}

It is easy to see that if $\eta_j\to \eta$ strictly, then $|\eta_j|\rightharpoonup^* |\eta|$ in $C_0(\Omega,\R^d)^*$. The above convergence prevents small-scale cancellations and concentration on the boundary. However, it does not prevent oscillation. To prevent oscillation, we must choose a "weight" $\mu\in \M^+(\Omega)$ and ask for convergence of $\eta_j$ on the graph $(\mu,\eta_j)$. We thus obtain a notion of $\mu$-strict convergence.

\begin{definition}
We say that $\eta_j\xrightarrow{\mu-s} \eta$ ($\eta_j$ converges $\mu$-strictly to $\eta$) if $\eta_j\to \eta$ weakly* in $C_0(\Omega,\R^d)^*$ and $(\mu,\eta_j)\xrightarrow{s} (\mu,\eta)$ in $C_0(\Omega,\R \times \R^d)^*$.
\end{definition}
Similarly to what was observed in the case of strict convergence, such a notion implies that $|(\mu,\eta_j)|\rightharpoonup^* |(\mu,\eta_j)|$ in $C_0(\Omega)^*$. Moreover, $\mu$-strict convergence of $\eta_j$ to $\eta$ simply amounts to weak* convergence and additional convergence of the following quantity: writing $\eta_j = \frac{\eta_j}{\mu} d\mu + \eta_j^{s,\mu}, \eta_j^{s,\mu} \perp \mu$,
\begin{align}
|(\eta_j,\mu)|(\Omega) = & \left| \left(\frac{\eta_j}{\mu} d\mu,\mu \right) + (\eta_j^{s,\mu},0 ) \right| \\
= & \int_\Omega \sqrt{1+ \left| \frac{\eta_j}{\mu} \right|^2} d\mu + |\eta_j^{s,\mu}|(\Omega) \to \int_\Omega \sqrt{1+ \left| \frac{\eta}{\mu} \right|^2} d\mu + |\eta^{s,\mu}|(\Omega) = |(\eta,\mu)|(\Omega).
\end{align}

When $\mu=\L^n$, we refer to such convergence as \textit{area-strict convergence}, in analogy with the area formula for smooth functions.

Reshetnyak continuity theorem (see \cite{reshetnyak1968weak} for the original) shows that strict convergence is equivalent to the convergence of 1-homogeneous functionals.

\begin{theorem}
Let $f(x,z) \in C(\Omega\times \R^d)$ be 1-homogeneous in $z$. If $\eta_j\to \eta$ strictly in the sense of measures, then
\begin{equation}
\int_\Omega f\left( x,\frac{\eta_j}{|\eta_j|} \right) d|\eta_j|(x) \to \int_\Omega f\left( x,\frac{\eta}{|\eta|} \right) d|\eta|(x) 
\end{equation}
\end{theorem}

In case $f$ is not 1-homogeneous but has an extension on the sphere compactification, we can obtain the following auto-convergence result by requiring $\mu$-strict convergence instead.

\begin{corollary}
Let $f\in \E(\Omega\times \R^d )$. If $\eta_j\xrightarrow{\mu-s} \eta$ in $C_0(\Omega,\R^d)^*$ then
\begin{equation}
\int_\Omega f\left( x,\frac{\eta_j}{\mu} \right)d\mu + f^\infty \left( x,\frac{\eta_j^{s,\mu}}{|\eta_j^{s,\mu}|} \right) d|\eta_j^{s,\mu}| \to \int_\Omega f \left( x,\frac{\eta}{\mu} \right) d\mu + f^\infty \left( x,\frac{\eta^{s,\mu}}{|\eta^{s,\mu}|} \right) d|\eta^{s,\mu}|.
\end{equation}
\end{corollary}

These are well-known results, but we write down a proof of the latter one because it gives an insight into how to move from one type of convergence to the other.

\begin{proof}
Consider the so-called \textit{perspective functional}
\begin{equation}
\tilde{f}(x,z,t) = \begin{cases}
f(x,\frac{z}{t})|t| \quad & t\neq 0 \\
f^\infty(x,z) & t=0
\end{cases}
\end{equation}
which is positively 1-homogeneous in the last variable. By the Reshetnyak continuity theorem, we know that, for $\eta\in \M(\Omega,\R^d)$,
\begin{equation}
\int_\Omega \tilde{f}(x,(\eta,\mu)) = \int_\Omega f \left( x,\frac{\eta}{d\mu} \right) d\mu + f^\infty \left( x,\frac{\eta^{s,\mu}}{|\eta^{s,\mu}|} \right) d|\eta^{s,\mu}|
\end{equation}
is sequentially continuous in the $\mu$-strict topology.
\end{proof}

Upon taking $f(x,z)=|z|$ we get that $\mu$-strict convergence implies strict convergence, for every $\mu\in \M^+(\Omega)$.

In light of these results, for a fixed measure $\mu\in \M^+(\Omega)$, the canonical embedding of measures $\eta\in \M(\Omega,\R^d)$ into the set of Young Measures on the sphere
\begin{align}
\eta\in \M(\Omega,\R^d) \mapsto \xi_\eta = \left( \delta_\frac{\eta}{\mu},|\eta^{s,\mu}|, \delta_{\frac{\eta^{s,\mu}}{|\eta^{s,\mu}|}} \right)
\end{align}
is sequentially $\mu$-strictly continuous. This is a solid justification for this choice of embedding. For the same reason, we can show why on larger compactifications we don't have, in general, a canonical choice.

\begin{theorem} \label{counter_example_embedding_YM}
Let $\mu\in \M^+(\Omega)$ with the property that there is $x\in supp(\mu\mres \overline{\Omega})$ with $\delta_x\perp \mu$, and let $f\in C(\R^d)$ of linear growth, $f\not\in \E_1(\R^d)$ (oscillate at infinity). There is $(u_j)_{j\in \N} \subset \D(\Omega,\R^d), u_j \xrightarrow{\mu-strictly} z\delta_x$ in $\M(\overline{\Omega},\R^d)$ for some $z\in \partial\B^d$, but
\begin{align}
\int_\Omega f(u_j(x))d\mu(x) \text{ does not converge}.
\end{align}
\end{theorem}

The above theorem implies that for all $e_{f_i,i\in \N}\geq e_f$ (in the sense of compactifications), $\xi_{u_j}$ does not converge in $Y(\mu,e_{f_i,i\in\N})$.

\begin{proof}
Because of the assumptions on $\mu$, we can find $x\in supp(\mu\mres \overline{\Omega})$ so that $\delta_x \perp \mu$. Notice that the result is unchanged if we instead consider $f(x) + C_1|z| + C_2$, so taking $C_1,C_2>0$ big enough we can assume that $f\geq 0$ everywhere. Find $z\in \partial \B^d$ and $z_j,w_j\to z$ so that
\begin{align}
\lim_n Tf(z_j) = M \text{ and } \lim_j Tf(w_j) = m \text{ exist, and } M>m.
\end{align}
Next, because $x\in supp(\mu)$ then $\mu(B_r(x))>0$ for all $r$. There are two possible scenarios. First, $x\in \Omega$, in which case we consider only those balls $B_r(x)$ so that $B_{2r}(x)\subset \Omega$. If $x\in \partial \Omega$ then we can $\delta_r \downarrow 0$ so that $\mu(B_r(x)) \cap \Omega_{-\delta(r)})>0$. Either way, we call $B_r(x)$ or $B_r(x)\cap \Omega_{-\delta(r)}$ simply $B_r$. Furthermore, we can find $\eps = \eps(r)>0$ so that
\begin{align}
B_r^\eps = (B_r)_\eps = \{x\in \R^n: d(x,B_r)<\eps\} \subset \Omega
\end{align}
and
\begin{align}
\lim_{r\downarrow 0} \frac{\mu( B_r^\eps) }{\mu(B_r) } = 1.
\end{align}
For each such $r$ find $\phi_r\in \D( B_r^\eps ), 0\leq \phi_r \leq 1$ so that $\phi_r(B_r)=1$. Put $u_r = \frac{\phi_r}{\int \phi_r d\mu}$. Clearly $u_r \rightharpoonup^* \delta_x$ in $C(\overline{\Omega})^*$. Next, refine the sequences $(z_j)_j$ and $(w_j)_j$ so that there is $r_j\downarrow 0$ so that
\begin{align}
\int_{B_{r_{2j}}}\phi_{2j} d\mu = 1-|z_j| \quad \text{ and } \quad \int_{B_{r_{2j+1}}^\eps} \phi_{2j+1}d\mu = 1-|w_j|,
\end{align}
and put
\begin{align}
u_j = \begin{cases}
u_{r_j} z_{\frac{j}{2}} & \text{ if j is even}, \\
u_{r_j} w_{\frac{j-1}{2}} & \text{ if j is odd}.
\end{cases}
\end{align}
First we show that $u_j\xrightarrow{\mu-strictly} z\delta_x$ in $\M(\overline{\Omega},\R^d)$. Because $z_j,w_j\to z$, it is enough to show that $u_r \xrightarrow{\mu-strictly} \delta_x$ in $\M(\overline{\Omega})$ as $r\downarrow 0$. Because $u_r\rightharpoonup^* \delta_x$ in $C(\overline{\Omega})^*$ then
\begin{align}
\liminf_{r\to 0} |(u_r d\mu,\mu)|(\overline{\Omega}) \geq |(\delta_x,\mu)|(\overline{\Omega}).
\end{align}
To achieve the opposite inequality, we calculate
\begin{align}
|(u_r d\mu,\mu)|(\Omega) = & |(0,\mu)|(\Omega\setminus B_r^\eps) + |(u_r,1)d\mu|(\Omega\cap B_r^\eps) = \mu(\Omega\setminus B_r^\eps) + \int_{B_r^\eps} \sqrt{1+ u_r^s}d\mu \\
= & \mu(\Omega\setminus B_r^\eps) + \frac{ \int_{B_r^\eps} \sqrt{ \left( \int \phi_r d\mu \right)^2+ \phi_r^s}d\mu }{ \int \phi_r d\mu } \\
\leq & \mu(\Omega\setminus B_r^\eps) + \frac{ \int_{B_r^\eps} \sqrt{ \mu(B_r^\eps)^2+ 1}d\mu }{ \mu(B_r) } \\
= & \mu(\Omega\setminus B_r^\eps) + \frac{ \mu(B_r^\eps) \sqrt{ \mu(B_r^\eps)^2+ 1} }{ \mu(B_r) } \to \mu(\Omega)+1 = |(\delta_x,\mu)|(\overline{\Omega}).
\end{align}
Next, we study how the integral behaves on alternating integers of the sequence $(u_j)_{j\in \N}$. If $j=2i$ then
\begin{align}
\liminf_i \int_\Omega f(u_{2i}(x))d\mu(x) \geq \liminf_i \int_{B_{r_2i}} f \left( \frac{z_i}{1-|z_i|} \right) d\mu = \lim_i Tf(z_i) = M.
\end{align}
On the other side, if $j=2i+1$ we get the upper bound
\begin{align}
\limsup_i \int_\Omega f(u_{2i+1}(x))d\mu(x) \leq \limsup_i \int_{B_{r_{2i+1}}^\eps} f \left( \frac{w_i}{1-|w_i|} \right) d\mu = \lim_i Tf(w_i) = m.
\end{align}
\end{proof}

\begin{remark}
The assumption on $\mu$ is sharp. If for all $x\in supp(\mu\mres \overline{\Omega})$ we have $\delta_x \not\perp \mu$, then all such $x$'s belong to $\Omega$. Consider the atomic decomposition of $\mu$:
\begin{align}
\mu = \mu^a + \mu^{n-a} = \sum_n \mu(x_n) \delta_{x_n} + \mu^{n-a},
\end{align}
see \refp{Appendix:atomic_decomposition}. If $\mu(\Omega\setminus \{x_n,n\in \N\})>0$ then we could find $x\in \Omega\setminus \{x_n,n\in \N\}$ and $\delta_x\perp \mu$, so that $\mu^{n-a} = 0$. Then $\mu = \sum_n \mu(x_n)\delta_{x_n}$, and the set $\{x_n,n\in \N\}$ contains its accumulation points. In particular $\{x_n,n\in \N\}=supp(\mu)$ is a compact subset of $\Omega$. It is easy to see that, in this setting, if $(\phi_j)_{j\in \N}\subset L^1(\mu)$ is bounded in norm and
\begin{align}
\phi_j \xrightarrow{Y(\mu,e_{f_i,i\in \N})} \big( \nu_x,\lambda,\nu_x^\infty \big),
\end{align}
then $\lambda \ll \mu$, i.e. $\lambda = \sum_n \lambda(x_n)\delta_{x_n}$ (because the space is countable and compact). Also
\begin{align}
\phi_j \rightharpoonup^* \phi = \left( \overline{\nu_x} + \overline{\nu_x^\infty} \frac{\lambda}{\mu} \right)d\mu
\end{align}
 in $C_0(X)^*$, $X= \{x_n,n\in \N\}$. Assume also that $\phi_j \to \phi$ $\mu$-strictly. This amounts to the following
\begin{align} \label{eq_optimal_no_embedding}
\int_X f(\phi_j)d\mu \to \int_X \langle \nu_x,f\rangle + \langle \nu_x^\infty,f^\infty\rangle \frac{\lambda}{\mu} d\mu = \int_X f(\phi)d\mu,
\end{align}
where $f(z)=\sqrt{1+|z|^2}$. $f$ is a strictly convex function, therefore the inequality $f(x+y)\leq f(x)+f^\infty(y)$ is strict unless $y=0$. We have
\begin{align}
\langle \nu_x,f\rangle + \langle \nu_x^\infty,f^\infty\rangle \frac{\lambda}{\mu} \geq & f(\overline{\nu_x}) + f^\infty \left( \overline{\nu_x^\infty} \frac{\lambda}{\mu} \right) > f\left( \overline{\nu_x} + \overline{\nu_x^\infty} \frac{\lambda}{\mu} \right) = f(\phi)
\end{align}
unless $\overline{\nu_x^\infty} = 0$ $\lambda$-a.e. So $\phi = \overline{\nu_x}$ $\mu$-a.e., and using convexity once again in \eqref{eq_optimal_no_embedding}, and the fact that $f^\infty = 1$, we get
\begin{align}
f(\phi) = & \langle \nu_x,f\rangle + \langle \nu_x^\infty,f^\infty\rangle \frac{\lambda}{\mu} \geq f(\overline{\nu_x}) + \frac{\lambda}{\mu} = f(\phi) + \frac{\lambda}{\mu}.
\end{align}
So $\lambda = 0$, which implies that the sequence $\phi_j$ does not concentrate. Moreover, $\phi(x)=\overline{\nu_x}$, which means that $\phi_j \to \phi$ in measure. Then $\phi_j\to \phi$ strongly in $L^1(\mu)$, and \refp{counter_example_embedding_YM} is false.
\end{remark}

\newpage
\section{Characterisation of gradient Young Measure\\ on general compactifications}

In this section, we show that generalised gradient Young Measures are characterised by a set of integral inequalities. A characterisation result was previously obtained in the context of the sphere compactification, see \cite{kristensen2010characterization} and \cite{kristensen2019oscillation} for the result on general differential operators, and it's here extended to general compactifications.

\subsection{Non-separability of the space of quasi-convex functions}

We start by showing that the set of quasi-convex functions having linear growth is non-separable. Lack of separability prevents sequential compactness and other essential properties that were used to develop the theory of generalised Young Measures (see the section above). Therefore, we are forced to consider only smaller countable collections of quasi-convex functions at the time, and cannot work with the entire class. To show that the class is non-separable, we modify the example by Muller \cite{muller1992quasiconvex} and generate quasi-convex functions that oscillate at different amplitudes in the same direction.

\begin{theorem}
The space of quasi-convex functions $f\colon\R^{2\times 2}\to \R$ having linear growth is not separable with respect to $\| T\cdot\|_{\infty,\B^{2\times 2}}$.
\end{theorem}

The idea behind the proof is to construct an uncountable family $\{ f_\Lambda \}_\Lambda$ so that $\|Tf_\Lambda-Tf_\Gamma\|_\infty \geq c$ for some universal constant $c>0$ and all $\Lambda\neq \Gamma$. We will split the proof of the above result into two different parts. First, we show that we have quasi-convex functions that oscillate along every possible sequence of natural numbers.

\begin{proposition} \label{prop_diff_seq}
There is $c>0$ such that for every $\Lambda\subset \N$ infinite so that $\Lambda^c$ is also infinite there exists $f_\Lambda\colon\R^{2\times 2}\to \R$ quasi-convex and having linear growth such that
\begin{enumerate}
\item $f_\Lambda (3^j \1)=0$ for all $j\in \Lambda$ and
\item $\frac{ f_\Lambda (3^j \1 ) }{ 3^j } \geq c$ for all $j\in \Lambda^c$ sufficiently large.
\end{enumerate}
\end{proposition}

To prove this theorem we first need a preliminary lemma, whose proof can be found in \citep{muller1992quasiconvex}{299}, lemma 4.
\begin{lemma} \label{ineq_muller}
For $k\in \R^+$ we let
\begin{align}
g_k\colon \R^{2\times 2}\to \R, F\mapsto  |F_{1,1}-F_{2,2}| + |F_{1,2} + F_{2,1}| + (2k - |F_{1,1} + F_{2,2}|)^+.
\end{align}
Then there exists $c_1>0$ such that
\begin{align}
Qg_k(0)\geq c_1 k
\end{align}
for all sufficiently large $k$s.
\end{lemma}

We can then prove \refp{prop_diff_seq}.
\begin{proof}
Let $\Lambda\subset \N$ as in the proposition above and set $f_\Lambda = Q g_\Lambda$ where
\begin{align}
g_\Lambda(F) = |F_{1,1}-F_{2,2}| + |F_{1,2}+F_{2,1}| + \inf\{ |F_{1,1}+F_{2,2}-2\cdot 3^i|,i\in \Lambda \}.
\end{align}
We have that $g_\Lambda(3^j \1)=0$ for all $j\in \Lambda$ and so $f_\Lambda(3^j \1)=0$ as well given that $g_\Lambda \geq 0$.



We first derive the following lower bound: compute
\begin{align}
g_\Lambda(3^j \1 + G) = |G_{1,1}-G_{2,2}| + |G_{1,2}+G_{2,1}| + \inf\{ |G_{1,1}+G_{2,2} + 2(3^j- 3^i)|: i\in \Lambda\};
\end{align}
because $3^j$ is increasing we estimate, for arbitrary $\beta\in \R$,
\begin{align}
\inf_{i\in \Lambda} |2(3^j -3^i) + \beta| \geq \left( \inf_{i\neq j} |2(3^j-3^i)| - |\beta| \right)^+ = \big( 2(3^j-3^{j-1}) - |\beta| \big)^+,
\end{align}
and so $g_\Lambda(3^j \1+G) \geq g_k(G)$ for all matrices $G$ and $k=3^j-3^{j-1}$, and $g_k$ given by
\begin{align}
g_k(F) = |F_{1,1}-F_{2,2}| + |F_{1,2} + F_{2,1}| + (2k - |F_{1,1} + F_{2,2}|)^+
\end{align}
does not depend on $\Lambda$. Apply \refp{ineq_muller} to find $c>0$ (independent of $\Lambda$) such that
\begin{align}
f_\Lambda(3^j \1) = Qg_\Lambda(3^j \1) \geq Qg_k(0) \geq c_1 (3^j-3^{j-1}) \text{ for j big enough.}
\end{align}
Dividing everything by $3^j$ we get
\begin{align}
\frac{ f_\Lambda(3^j \1) }{ 3^j } \geq \frac{2}{3}c_1 \equiv c.
\end{align}
\end{proof}

It's not a priori clear if these functions are "far away from each other at infinity", as subsets of natural numbers could intersect infinitely many times. To show that there is a wide variety of sequences that differ at infinity, we define the following relation.
\begin{definition}
Given $\Gamma,\Lambda$ any two sequences (not necessarily subsets of $\N$), we say that $\Gamma \leq \Lambda$ provided $\Gamma$ is eventually a subset of $\Lambda$, i.e.
\begin{align}
& \Lambda = (a_i)_{,i\in \N},\Gamma = (b_i)_{i\in \N}, \quad \Lambda \leq \Gamma \quad \iff\\
& \qquad \quad \text{ there exists }k>0: (a_i)_{i\geq k} \text{ is a subsequence of } ( b_i)_{i\geq 0}.
\end{align}
Let $[\Lambda]$ be the equivalence class of $\Lambda$ with respect to $\leq$, i.e. $\Gamma\in [\Lambda]$ if $\Gamma\leq \Lambda \leq \Gamma$.

If $\Lambda'\in [\Lambda]$ and $\Gamma'\in [\Gamma]$ then $\Lambda\leq \Gamma$ if and only if $\Lambda'\leq \Gamma'$.

We use $\big( \F,\leq \big)$ to indicate the set of equivalence classes with the inherited order.
\end{definition}

The above ordering is needed because, to show that we have uncountably many sequences that are independent of each other at infinity, we will use Zorn's lemma and find a maximal set. One could also reason that the Stone-Cech compactification of natural numbers is not metrisable and reason by contradiction using a suitably adapted version of \refp{subAlgebraRepresentation}, but we decided to not pursue this path.

\begin{lemma}
There exists an uncountable family $G\subset F$ such that for every different pair $\Lambda,\Gamma\in G$, $\Lambda$ is not comparable to either $\Gamma$ nor $\Gamma^c$ with respect to $\leq$.
\end{lemma}

The above means that we can find a set $G$ such that given any two sequences of natural numbers in $\Lambda,\Gamma\in G$, one is frequently in the other sequence and its complement, i.e. $\Lambda\cap \Gamma$ and $\Lambda \cap \Gamma^c$ are both infinite.

\begin{proof}
Consider the set of subsets
\begin{align}
\mathcal{G} = \{G\subset \F: \text{no pair within $G$ is comparable according to} \leq \},
\end{align}
ordered by inclusion $\subset$. The above set is non-empty, which can be seen by taking $\Lambda = 2\N$ and $\Gamma = 4\N \cup (4\N+1)$. Every chain in $\mathcal{G}$ has an upper limit given by its union. By Zorn's lemma, there exists a maximal element. I claim that the maximal element has uncountably many elements. To prove the claim we first assume that the maximal element $\overline{G}\subset \mathcal{G}$ is countable or finite and show that it is always possible to extract an extra incomparable element.

To do so, I will show that given any countable or finite collection of infinite natural numbers $\{ c_i^j,i\in\N \}_{j\in\N}$ there is $\{c_i,i\in\N\}$ such that $\{c_i,i\in\N\} \cap \{c_i^j,i\in\N\}$ is infinite for all $j$ and $\{c_i,i\in\N\} \cap \{c_i^j,i\in\N\}<\{c_i^j,i\in\N\}$ according to the order previously established. Consider the isomorphism
\begin{align}
N\colon \F \to \{ 0,1 \}^\N, \{ c_i, i\in\N \} \mapsto \left( Nc_i = \begin{cases}
1 & \text{ if } i\in \{ c_k,k \in\N \} \\
0 & \text{ otherwise }
\end{cases} \right)_{i\in\N}
\end{align}
Practically speaking, we are replacing subsets of natural numbers to sequences that take values $1$ when the $i$-th number is in the set, $0$ otherwise.

Set initially $Nc_i=0$ for all $i$. At $i_1$ so that $Nc^1_{i_1}=1$ for the first time put $Nc_{i_1}=1$. We then iterate "diagonally" in the following way. At the $n$-th iteration find $i_{n+1}$ so that $Nc_{k_j}^j=1$ for all $1\leq j\leq n$ and some $i_n< k_{j-1} < k_j$ and $Nc^j_{t_j}=1$ for all $1\leq j\leq n+1$ and $k_n < t_{j-1} < t_j < t_{n+1} = i_{n+1}$. Set $Nc_{t_j}=1$ for all $1\leq j\leq n+1$. This procedure stops if the maximal set is finite, otherwise can be iterated countably many times.

This way we guarantee that $Nc_{k_j}=0$ for $i_n<k_j$ and $Nc_{k_j}^j=1$, which means that $Nc_i$ skips infinitely many $1$s from each sequence $(Nc_i^j)_{i\in \N}$, for all $j\in \N$. On the other side $Nc_{t_j}=1 = Nc^j_{t_j}$, $t_j\leq i_{n+1}$, so $Nc_i$ is also frequently in every sequence $(Nc_i^j)_{i\in \N}$.

Going back to our countable maximum element $\overline{G} = \big\{ \{ a_i^j,i\in\N \}, j\in\N \big\}$, we can apply the previous construction to find $\{c_i,i\in\N\} \in F$ generated by the countable family
\begin{equation}
\{ c_i^j,i\in \N\}_{j\in \N} = \Big\{ \{ a_i^j,i\in\N \}\ , \ \N\setminus \{ a_i^j,i\in\N \} \Big\}_{ j\in\N }
\end{equation}
Because $\{c_i,i\in\N \}$ is frequently and properly in $\{ a_i^j,i\in\N \}$ and its complement $\N\setminus \{ a_i^j,i\in\N \}$ for all $j$, then $\{c_i,i\in\N\}$ is not comparable to any member of the family and this contradicts the maximality of our set.
\end{proof}

We are now ready to prove \refp{prop_diff_seq}.

\begin{proof}
By the lemma we can find an uncountable set of uncomparable subsequences of $N$, call it $\overline{G}$. I claim that if $\Lambda,\Gamma\in \overline{G}, \Lambda\neq \Gamma$ then $f_\Lambda$ and $f_\Gamma$ have different recessions at infinity. Find $\{x_n\} \in e_{f_\Lambda}$ with $\{x_n\} \geq \{3^j\1, j\in \Lambda\}$, where the inequality could be strict given that there might be more zero points at infinity. Given our construction, we immediately have that $\{x_n\}\not\geq \{3^j \1,j\in \N\setminus \Lambda\}$ as $f_\Lambda(3^j\1)\geq c3^j$ for all $j\in \N\setminus \Lambda$. Also, $\Gamma$ intersects $\N\setminus \Lambda$ and $\Lambda$ infinitely many times, and vice versa, so that
\begin{align}
\limsup_n \frac{ f_\Gamma(x_n) }{ |x_n| } \geq \limsup_{j\in \Lambda \cap \Gamma} \frac{ f_\Gamma(3^j \1) }{ 3^j } \geq c, \quad \text{and} \quad \liminf_n \frac{ f_\Gamma(x_n) }{ |x_n| } \leq \liminf_{j\in \Lambda \cap \Gamma} \frac{ f_\Gamma(3^j \1) }{ 3^j } = 0,
\end{align}
which shows that $\{x_n\}\not\in e_{f_\Gamma}$. In terms of the non-separability, by the definition of $x_n$, we have
\begin{align}
\lim_n \frac{ f_\Lambda(x_n) }{ |x_n| }=0,
\end{align}
thus
\begin{align}
\|Tf_\Lambda-Tf_\Gamma\|_\infty \geq \limsup_n \left| \frac{ f_\Gamma(x_n) }{ |x_n| } - \frac{ f_\Lambda(x_n) }{ |x_n| } \right| = \limsup_n \left| \frac{f_\Gamma(x_n)}{|x_n|} \right| \geq c,
\end{align}
where $c$ is independent of $\Lambda$ or $\Gamma$.
\end{proof}

Given that the space is metric, having such a property prevents separability. We can end this section with the following corollary that incorporates higher dimensions:

\begin{corollary}
The set of quasi-convex functions having linear growth $f\colon \R^{m\times n}\to \R$ is separable in the topology induced by $\|T(\cdot)\|_\infty$ if and only if $min(m,n)=1$.
\end{corollary}

\begin{proof}
If $m$ or $n$ is 1, quasi-convex functions are convex and therefore the space is separable. This is because convex functions admit a limit at infinity in every direction; in this case, we actually recover the sphere compactification.

On the other side, for a function $g\colon \R^{2\times 2}\to \R$ we let
\begin{align}
gP \equiv g\circ P\colon\R^{n\times m} \to \R,\quad P\colon\R^{m\times n}\to \R^{2\times 2}, M\mapsto \begin{bmatrix}
M_{1,1}, M_{1,2} \\ M_{2,1}, M_{2,2}
\end{bmatrix}.
\end{align}
If $g$ is quasi-convex and locally bounded we let $\phi\in \D(Q,\R^m)$ and compute
\begin{align}
\int_Q gP(D\phi + z) = \int_{Q\subset \R^{n-2}} d\L^{n-2} \int_{[0,1]^2} dx_1x_2 g(PD\phi + Pz) \geq \int_{Q\subset \R^{n-2}} d\L^{n-2} gP(z) = gP(z),
\end{align}
i.e. $gP$ is quasi-convex. Then the set $f_\Lambda P$ is uncountable and
\begin{align}
\| T( f_\Lambda P - f_\Gamma P)\|_{\infty,\B^{m\times n}} = \|T(f_\Lambda - f_\Gamma)\|_{\infty,\B^{2\times 2}} \geq c
\end{align}
if $\Gamma\neq \Lambda$, so the space cannot be separable.
\end{proof}

Notice that separability is important to achieve both the Young Measure representation and for sequential compactness in the inherited weak star topology of Young Measures.

\subsection{Characterisation of Gradient Young Measures}

In this section, we characterise Gradient Young Measures (on separable compactification) via certain Jensen-like integral inequalities.

It is worth mentioning that the result for the sphere compactification, achieved in \cite{kristensen2010characterization}, can be easily improved in consideration of the fact that $\limsup_{t\to \infty} f(tz)t^{-1}$ is 1-homogeneous rank one convex, and so convex at points $rank(z)=1$ (see \citep{kirchheim2016rank}{528})). In what follows, we cannot use this type of auto-convexity. In our context, $f^\infty$ lives on a general compactification and convexity at points of rank one, as a Jensen's type inequality, is not necessarily true.

To prove our result, we will adopt the same strategy as in \cite{kristensen2019oscillation}.

Preliminarily to stating the theorem, we define the upper recession of a function relative to a general compactification.
\begin{definition}
Let $e_{f_i,i\in\N}$ be a separable compactification. For any $f$ having linear growth, we define
\begin{align}
f^{\sharp,e_{f_i,i\in\N}}((z_n)) = \sup_{(w_n)_n\in [(z_n)_n]} \limsup_n Tf(w_n),
\end{align}
where $(w_n)_n$ are sequences belonging to the equivalence class of $(z_n)_n$ within $e_{f_i,i\in \N}$.
\end{definition}

The reason why we introduce this notion is that the strategy for proving the characterisation theorem makes use of the trivial fact that $f\geq f^{qc}$, $f^{qc}$ being the quasi-convex envelope of $f$ (see, for example, \cite{dacorogna2007direct})
\begin{align}
f^{qc}(z) = \inf_{\phi\in \D(Q)} \int_Q f(z+D\phi(x))dx.
\end{align}
However, $f^{qc}$ does not need to live in the same class of separable quasi-convex functions, so the implication
\begin{align}
\lim_n \frac{f(z_n)}{|z_n|} \text{ exists } \Rightarrow \lim_n \frac{f^{qc}(z_n)}{|z_n|} \text{ exists}
\end{align}
could be false for some functions $f$.

Fix any compactification $e_{f_i,i\in\N}$ and a function $Tg \in e_{f_i,i\in\N}$. If $g\geq f$ then
\begin{align}
g^\infty((z_n))= \lim_n Tg(z_n) \geq \sup_{(z_n)\in [(z_n)]} \limsup_n Tf(z_n) = f^{\sharp,e_{f_i,i\in\N}}((z_n)).
\end{align}
$f^{\sharp,e_{f_i,i\in\N}}$ does not need to be continuous on $e_{f_i,i\in\N}$ with respect to its product topology. However, we can show that it is still upper semi-continuous.

\begin{lemma}
Let $e_{f_i,i\in \N}$ be a separable compactification and $f$ a function having linear growth, then $f^{\sharp,e_{f_i,i\in \N}}$ is upper semi-continuous on $\partial e_{f_i,i\in \N}$.
\end{lemma}

\begin{proof}
Notice that $\partial e_{f_i,i\in \N}$ is metrisable, so it is enough to show sequential upper semi-continuity. Let $(z_n^j)_n\in \partial e_{f_i,i\in \N}$ so that $(z_n^j)_n \xrightarrow{j} (z_n)_n$. By the very definition of $g^{\sharp,e_{f_i,i\in \N}}( (z_n^j)_n)$, for fixed $\eps>0$ we can find $n_j \geq k_j$, where $k_j$ is a natural number to be selected, so that $g^{\sharp,e_{f_i,i\in \N}}( (z_n^j)_n) \leq \eps + Tg( (z^j_{n_j}) )$, where $(z^j_{n_j})$ is a constant sequence and belongs to $e_{f_i,i\in \N} \setminus \partial e_{f_i,i\in \N}$. We want to show that we can select $k_j$ so that $(z_{n_j}^j)_j$ belongs in the equivalence class of $(z_n)_n$. By applying the dominated convergence theorem we get
\begin{align}
\lim_j \sum_i \lim_n 2^{-i} |f_i(z_n^j)-f_i(z_n)|= 0 = \lim_j \lim_n \sum_i 2^{-i} |f_i(z_n^j)-f_i(z_n)|,
\end{align}
and so can find $k_j$ so that
\begin{align}
\sum_i 2^{-i} |f_i(z_n^j)-f_i(z_n)| \leq \eps_j \quad \forall n\geq k_j,
\end{align}
where $0\leq \eps_j \downarrow 0$. This shows that the above sequence $(z_{n_j}^j)_j \in [ (z_n)_n ]$ (the equivalence class), thus
\begin{align}
\limsup_j g^{\sharp,e_{f_i,i\in \N}} ( (z_n^j)_n ) \leq \eps + \limsup_j Tg( (z^j_{n_j}) ) \leq \eps + g^{\sharp,e_{f_i,i\in \N}}( (z_n)_n).
\end{align}
By the arbitrariness of $\eps>0$ we conclude upper semi-continuity of $g^{\sharp,e_{f_i,i\in \N}}$
\end{proof}

In particular, $g^{\sharp,e_{f_i,i \in \N}}$ is Borel measurable on $e_{f_i,i\in \N}$. The above statement can be generalised to extensions of functions over more general compact metric spaces, but this version suffices for our scopes.

We are interested in studying those Young Measures that are generated by gradients. So we define the following.
\begin{definition}
We say that $\nu\in Y(e_{f_i,i\in\N})$ is a (generalised) gradient Young Measure if there exists a sequence $u_j\in BV(\Omega,\R^m)$ such that
\begin{align}
Du_j \xrightarrow{Y(e_{f_i,i\in\N})} \big( \nu_x,\lambda,\nu_x^\infty \big).
\end{align}
We use $GY(e_{f_i,i\in\N})$ to refer to these subsets of Young Measures.
\end{definition}

The convergence of measure derivatives has not been fully comprehended yet and it is still the subject of active research. This means that it is not so clear how rich the above class is, and with which frequency gradients oscillate - or at least within the setting of weak* convergence of measures.

\begin{remark}
We can use the characterisation lemma for Young Measure on the sphere to show that the class is still quite vast. Indeed, by \citep{kristensen2010characterization}{541} Theorem 1, fix any $z=a\otimes b$ and $\nu^\infty\in \P(\partial \B)$ with $\overline{\nu^\infty}=z$. Then gradient Young Measure on the sphere
\begin{equation}
\big( \delta_0, \H^{n-1}\mres (B\cap b^\perp),\nu^\infty \big)
\end{equation}
satisfies the characterisation theorem from \cite{kristensen2010characterization}, with $u=a\chi_{x\cdot b\geq 0}$ and so it is generated by a sequence of gradients $Du_j\in BV( B_1(0)), B_1(0)\subset \R^n$. Because $Du_j$ is bounded in $BV$, we can find a subsequence $(u_{j_k})_{k \in \N}$ such that
\begin{equation}
Du_{j_k} \xrightarrow{Y(e_{f_i,i\in \N}), \text{ as } k\to \infty} \big( \delta_0,\H^{n-1}\mres (B\cap b^\perp), \eta^\infty \big),
\end{equation}
with clearly $\pi_{\partial \B} \eta^\infty = \nu^\infty$. Using \refp{disintegration_compactification_Y} to write $\eta^\infty = P_z d\nu^\infty, z\in \partial \B^d$, it remains an open question to understand how many probabilities $P_z$ over subsequences $z_n\to z$ can be generated by gradients.
\end{remark}

We now state the main theorem of this section.
\begin{theorem} \label{characterisation_general_cpt}
Let $\Omega\subset \R^n$ be a bounded Lipschitz domain and $e_{f_i,i\in\N}$ be a separable compactification of quasi-convex functions and consider a generalised Young Measure $\nu\in Y(e_{f_i,i\in\N})$ that satisfies $\lambda(\partial \Omega) = 0$.

Then $\nu\in GY(e_{f_i,i\in\N})$ is a Young Measure generated by a sequence
\begin{align}
(\phi_j\star (Du\mres \Omega) + Du_j) \xrightarrow{Y(e_{f_i,i\in\N})} \big( \nu_x,\lambda,\nu_x^\infty \big),
\end{align}
where $u\in BV(\Omega,\R^m)$, $u_j\in \D(\Omega,\R^m)$ and $\|u_j\|_1\to 0$, and $\phi_j$ is any sequence of mollifiers with $\phi_j\rightharpoonup^* \delta_0$, if and only if there is $u\in BV(\Omega,\R^m)$ such that
\begin{enumerate}
\item $\ll 1\otimes |\cdot |,\nu\gg <+ \infty$, and for all $f$ quasi-convex and having linear growth,
\item $f(\nabla u(x))dx \leq \langle \nu_x,f\rangle dx + \langle \nu_x^\infty,f^{\sharp,e_{f_i,i\in\N}} \rangle \frac{\lambda}{\L^n} dx$ and
\item $f^\infty(D^s u) \leq \langle \nu_x^\infty,f^{\sharp,e_{f_i,i\in\N}} \rangle d\lambda^s$.
\end{enumerate}
\end{theorem}

We can adjust the above theorem to fix the boundary of the converging sequence so that it's always equal to $u$ in the sense of trace.

\begin{lemma}
If $\nu\in Y(e_{f_i,i\in\N})$ is generated by a sequence
\begin{align}
\phi_j \star (Du\mres \Omega) + Du_j
\end{align}
as above, then there exists another sequence $v_j\in C^\infty(\Omega)\cap W_u^{1,1}(\Omega)$ such that
\begin{align}
D(v_j + u_j) \xrightarrow{Y(e_{f_i,i\in\N})} \nu.
\end{align}
In particular, $\nu\in GY(e_{f_i,i\in \N})$.
\end{lemma}

\begin{proof}
We can find $u_j\to u$ strictly in $BV(\Omega)$ with $u_j\in C^\infty(\Omega)\cap W^{1,1}_u(\Omega)$, see for example \cite{kristensen2010characterization} Lemma 1 for a proof of this fact. In the construction of the $u_j$s just mentioned, it is possible to select $\phi_j\star (Du\mres\Omega)$ on $\Omega_{-\eps}$ for $j$ big enough, $\phi_j$ as in \refp{characterisation_general_cpt}. Also, without loss of generality, we can assume that $|Du|(\partial \Omega_{-\eps}) = |Du_j|(\partial \Omega_{-\eps}) =0$. Because the $f_i$s are all Lipschitz, it is enough to test against $f\in Lip(\R^{m\times n})$ with $Lip(f)\leq 1$. We then compute
\begin{align}
\int_\Omega |f(\phi_j\star (Du\mres \Omega)+Dv_j)- f(Du_j+Dv_j)|dx \leq & \int_\Omega |\phi_j\star (Du\mres \Omega) - Du_j| \\
\leq & \int_{\Omega\setminus \Omega_{-\eps}} |\phi_j\star (Du\mres \Omega)| + |Du_j| = I_j + II_j
\end{align}
By strict convergence of both integrands, we have that
\begin{align}
\limsup_j I_j + II_j \leq 2 |Du|(\Omega\setminus \Omega_{-\eps}),
\end{align}
and so use a diagonal argument to conclude the existence and equality of the limit Young Measure.
\end{proof}

It's implicit in \refp{characterisation_general_cpt} that
\begin{align}
Du = \overline{\nu} = \langle\nu_x,\cdot\rangle dx + \langle \nu_x^\infty,\cdot\rangle d\lambda.
\end{align}
Also, the above inequalities can be written in the sense of distribution, in the form
\begin{align}
& \int_\Omega \phi(x) \langle\nu_x,f\rangle dx + \int_\Omega \phi(x) \int_{\partial e_{f_i,i\in\N}} f^{\sharp,e_{f_i,i\in\N}} d\nu_x^\infty d\lambda \\
& \qquad \geq \int_\Omega \phi(x)f(\nabla u(x))dx + \phi(x)f^\infty \left( \frac{D^su}{|D^su|} \right) d|D^su|
\end{align}
for all $\phi\in \D(\Omega),\phi\geq 0$.

To prove the characterisation result we will follow the same strategy as in \cite{kristensen2019oscillation}. We initially prove the result for homogeneous gradient Young Measures and then extend the theorem to the inhomogeneous case. Notice that Young Measures that act on functions $f=f(z)$ that only depend on $z$ can be represented by
\begin{align}
\int_\Omega \langle \nu_x,f\rangle dx + \int_{\overline{\Omega}} \langle \nu_x^\infty,f^\infty\rangle d\lambda = \int_\Omega fd\nu^0 + \int_{\partial e_{f_i,i\in\N}} f^\infty d\nu^\infty,
\end{align}
where for $e_{f_i,i\in\N}$ the separable compactification that extends $f$,
\begin{align}
\nu^0 = \nu_x d\L^n \in \M^+(\Omega) \quad \text{and} \quad \nu^\infty = \nu_x^\infty d\lambda \in \M^+(\partial e_{f_i,i\in\N}).
\end{align}
The (push-forward) Kantorovich metric is then
\begin{align}
\| (\nu^0,\nu^\infty) \|_K = \sup_{\Phi\in H,\|T\Phi\|_{Lip(e_{f_i,i\in\N})} \leq 1} \left| \int_{\R^d} \Phi d\nu^0 + \int_{e_{f_i,i\in\N}} \Phi^\infty d\nu^\infty \right|.
\end{align}

For $z\in \R^d$ we let $\mathcal{Y}$ be the set of pairs $\big( \nu^0,\nu^\infty \big)\in \M_1^+(\R^d)\times \M^+(e_{f_i,i\in\N})$ such that there is a sequence $u_j\in \D(Q,\R^m)$, where $Q$ is the unit cube, so that $z+Du_j\xrightarrow{Y(e_{f_i,i\in\N})} \big( \nu^0,\nu^\infty \big)$ and $\|u_j\|_1\to 0$. The following proposition follows from obvious variations of the proofs contained in \citep{kristensen2019oscillation}{8}, lemmas 3.7,3.8,3.9. The proofs are essentially the same as they only use the separability of the compactification.
\begin{lemma}
The family $\{ \eps_{z+Du}:u\in \D(Q,\R^m)\}$ is weakly* dense in $\mathcal{Y}$, and $\mathcal{Y}$ is a weak* closed and convex subset of homogeneous Young Measures.
\end{lemma}

We can now prove the main theorem in case $\big( \nu^0,\nu^\infty \big)$ is a homogeneous gradient Young Measure.

\begin{proposition} \label{result_homogeneous}
Let $\nu = \big( \nu^0,\nu^\infty \big) \in M^+_1(\Omega)\times \M^+(\partial e_{f_i,i\in\N})$ and $z\in \R^{m\times n}$. Then $\nu \in \mathcal{Y}$ if and only if there is $z\in \R^{m\times n}$ such that
\begin{align}
\int_{\R^{m\times n}} fd\nu^0 + \int_{\partial e_{f_i,i\in\N}} f^{\sharp,e_{f_i,i\in\N}} d\nu^\infty \geq f(z)
\end{align}
for all $f\colon\R^{m\times n}\to \R$ quasi-convex and having linear growth.
\end{proposition}

\begin{proof}
Suppose that $\nu\in \mathcal{Y}$ and let $z+Du_j, u_j\in \D(Q,\R^m)$ be the generating sequence, i.e. for all $\Phi\in T^{-1} e_{f_i,i\in\N}$,
\begin{align}
\int_Q \Phi(z+Du_j)dx \to \langle \nu,\Phi\rangle = \int_{\R^{m\times n}} \Phi d\nu^0 + \int_{\partial e_{f_i,i\in\N}} \Phi d\nu^\infty.
\end{align}
Fix an arbitrary $f$ having linear growth and quasi-convex and let $e_{f,f_i,i\in\N}$ the bigger compactification. Upon extracting a subsequence we have that
\begin{align}
z+Du_j \xrightarrow{Y(e_{f,f_i,i\in\N})} \big( \nu^0,\tilde{\nu}^\infty \big),
\end{align}
where we identify the gradient Young Measure with its tensor products as $f=f(z)$. By quasi-convexity, we have
\begin{align}
\int_{\R^{m\times n}} fd\nu^0 + \int_{\partial e_{f,f_i,i\in\N}} f^\infty d\tilde{\nu}^\infty = \limsup_j \int_Q f(z+Du_j)dx \geq f(z).
\end{align}
On the other side, using the decomposition of angle concentration Young Measure \refp{disintegration_compactification_Y} we also obtain that
\begin{align}
\int_{\partial e_{f,f_i,i\in\N}} f^\infty d\tilde{\nu}^\infty = \int_{\partial e_{f_i,i\in\N}} \int f^\infty dP_{(z_n)_n} d\nu^\infty \leq \int_{\partial e_{f_i,i\in\N}} f^{\sharp,e_{f_i,i\in\N}} d\nu^\infty.
\end{align}

For the other implication, because $\mathcal{Y}$ is weakly* closed and convex, we can write $\mathcal{Y}=\cap H$ where $H$ are half-spaces containing $\mathcal{Y}$, which can be written as
\begin{align}
H = \{l\in H^*:l(\Phi) \geq t\}.
\end{align}
In particular, we can test the above inequality against $\eps_{z+Du}$ and get
\begin{align}
t\leq \eps_{z+Du}(\Phi) \leq \int_Q \Phi(z+Du)dx
\end{align}
for all $u\in \D(Q,\R^m)$. Passing to the infimum over all such $u$s we deduce $t\leq \Phi^{qc}(z)$ and so
\begin{align}
\langle \nu,\Phi\rangle = & \int_{\R^{m\times n}} \Phi d\nu^0 + \int_{\partial e_{f_i,i\in\N}} \Phi^\infty d\nu^\infty \\
\\ \geq & \int_{\R^{m\times n}} \Phi^{qc} d\nu^0 + \int_{\partial e_{f_i,i\in\N}} (\Phi^{qc})^{\sharp,e_{f_i,i\in\N}} d\nu^\infty \geq \Phi^{qc}(z) \geq t
\end{align}
which shows that $\nu\in H$.
\end{proof}

\subsubsection{Inhomogenization}

In what follows, we will prove a semi-approximation result for the absolutely continuous and singular parts separately and then put them together via \refp{join_sequence_Y}. In each case, we will use a covering argument to boil it down to the homogeneous case, which was solved in the above section.

Consider a standard mollifier $\phi_t(x) = t^{n-1} \phi(\frac{x}{t})$, where $\phi\in \D(Q)$ and let $M=\|D\phi\|_\infty$. Also, unless otherwise specified, the norm on $\R^n$ is the maximum norm $\|x\|=\max_i |x_i|$.

\begin{lemma}
Given $\eps>0$ there is $t_\eps>0$ and a family $\varphi_t\in \D(\Omega,\R^m)$ with $\|\varphi_t\|_1 \leq \eps$ so that
\begin{align}
\left| \int_\Omega \eta \Phi(0) + \eta \langle \Phi^\infty,\nu_x^\infty\rangle d\lambda^s - \int_\Omega \eta \Phi(\phi_t\star (\overline{\nu^\infty} d\lambda^s) + D\varphi_t)dx \right| < \eps
\end{align}
for all $t\in (0,t_\eps)$ uniformly in $\|\eta\|_{Lip} \leq 1$ and $\|T\Phi\|_{Lip,graph(f)}\leq 1$.
\end{lemma}

The idea behind this approximation result is the following. The singular part of the centre of mass (which is just $Du\in \M(\Omega,\R^d)$) is approximated by mollification. Such a procedure generates area-strictly convergent smooth approximations. At the same time, we generate angle concentration and oscillation via compactly supported functions. Because the first type of convergence is very strong, and the latter doesn't concentrate, the two modes of convergence don't interfere with each other. Notice that this strategy would not be possible using the bare notion of weak* convergence because of the lack of quantifiability, whereas the (equivalent in this case) Kantorovich metric gives us an "exact" quantity to approximate.

Before proving the above statement we remind that, according to \refp{Lipschitz_pullback}, $T$ pulls back bounded sets of Lipschitz functions on $e_{f_i,i\in\N}$ to bounded sets of Lipschitz functions on $\R^{m\times n}$ (provided $f_i$ are Lipschitz). Therefore, all the functions in the following theorem can be taken to be, after renormalisation, $1$-Lipschitz in both spaces.

From now on, after fixing a compactification, we will always identify
\begin{align}
\|T\Phi\|_{Lip} = \|T\Phi\|_{Lip(e_{f_i,i\in\N})} = \|T\Phi\|_\infty + \sup_{x\neq y} \frac{|T\Phi(x)-T\Phi(y)|}{|x-y| + \sum_i 2^{-i} |Tf_i(x)-Tf_i(y)|}.
\end{align}

\begin{proof}
Fix $\eps>0$ and apply Luzin's theorem to the $\lambda^s$ map
\begin{align}
x\in \overline{\Omega} \to (\delta_0,\nu_x^\infty)\in \M_1^+(\R^d)\times \M^+(\partial e_{f_i,i\in\N}) \hookrightarrow \big( (T^{-1} e_{f_i,i\in\N} )^* \big)^+
\end{align}
to find a compact set $C=C_\eps\subset \overline{\Omega}$ with $\lambda^s(\overline{\Omega}\setminus C)<\lambda^s(\overline{\Omega})\eps$ restricted to which the above map is uniformly continuous, with modulus of continuity $\omega=\omega_\eps$. Without loss of generality, assume that $\L^n(C^s)=0$ and because $\lambda^s(\partial \Omega)=0$ then
\begin{align}
\Delta=\Delta_\eps = d(C^s,\partial \Omega)>0.
\end{align}
For the moment, fix two integers $a,b\in \N$ and put $t=2^{-a}$, so that $\phi_t>0$ if and only if $\|x\|<t$. Let $a$ be so large that
\begin{align}
2t \leq \Delta \quad \text{ and } \quad a \geq \log_2 \left( \frac{2}{\Delta} \right).
\end{align}
Denote by $\F$ the collection of $a+b$-th generation dyadic cubes $\Q$ in $\R^n$ so that $d(\Q,\partial \Omega)>2^{-a}$, and for each such $\Q\in \F$ we define
\begin{align}
r_\Q = \fint_\Q \phi\star (\lambda^s\mres C^s)dx.
\end{align}
Notice that $r_\Q>0$ means that $dist(\Q,C^s)<t$, and so for each such $\Q$ we can find $x_\Q\in C^s$ so that $d(x_\Q,\Q)<t$. Denote by $\F^s$ the set of those $\Q\in\F$ for which $r_\Q>0$. In particular, if $\Q\in \F^s$ we can find $x_\Q\in C^s$ so that $\sup_\Q \|x-x_\Q\|<2t$.

For every quasi-convex function having linear growth we have
\begin{align}
f(z+w)\leq f(z)+f^\infty(w)
\end{align}
for all $z\in \R^{m\times n}$ and $rank(w)=1$, see \citep{kirchheim2016rank}{536}, lemma 2.5 (we don't need regular recession for this result to hold). Then by assumption, we have
\begin{align}
f(r_\Q \overline{\nu^\infty_{x_\Q}}) \leq f(0) + r_\Q f^\infty(\overline{\nu^\infty_{x_\Q}} ) \leq f(0) + r_\Q \int_{\partial e_{f_i,i\in\N}} f^{\sharp,e_{f_i,i\in\N}} d\nu_{x_\Q}^\infty
\end{align}
for all $f$ quasi-convex and having linear growth.

Going back to the homogeneous case \refp{result_homogeneous}, we can select $\varphi^\Q\in \D(\Q,\R^m)$ with $\|\varphi^\Q\|_1 < \eps \lambda^s(\Q)$ such that
\begin{align}
\| (\delta_0,\nu_{x_\Q}^\infty r_\Q) - \eps_{r_\Q \overline{\nu^\infty_{x_\Q} } + D\varphi^\Q} \|_K < \eps.
\end{align}
Define $\varphi = \sum_{\Q\in \F^d} \varphi^\Q \in \D(\Omega,\R^m)$ and $\|\varphi\|_1 \leq \eps \lambda^s(\Omega)$. The sought-after map is then
\begin{align}
\xi^s = \phi\star (\overline{\nu^\infty_x} d\lambda^s + D\varphi) \in \D(\R^n,\R^d).
\end{align}

To prove that this function is the desired one, we fix $\| \eta\|_{Lip} \leq 1, \| \Psi\|_{Lip(e_{f_i,i\in\N})}\leq 1$ as in the assumptions. We have
\begin{align}
\int_\Omega \eta \langle \Phi^\infty,\phi\star (\overline{\nu^\infty} d\lambda^s)\rangle dx = \int_\Omega \eta \langle \Phi^\infty,\phi\star (\overline{\nu^\infty} d\lambda^s\mres C^s)\rangle dx + \overbrace{\int_\Omega \eta \langle \Phi^\infty,\phi\star (\overline{\nu^\infty} d\lambda^s\mres \Omega\setminus C^s)\rangle dx}^{=\mathcal{E}_1},
\end{align}
and $\ |\mathcal{E}_1|\leq \eps \lambda^s(\Omega)$. Notice that here
\begin{align}
\int_\Omega \eta \langle \Phi^\infty, \phi\star (\nu^\infty d\lambda^s \mres U)\rangle dx = \int_\Omega \eta(x) \int_U \phi(x-y) \int_{\partial e_{f_i,i\in\N}} \Phi^\infty d\nu_y^\infty d\lambda^s(y)dx
\end{align}
where $U=\Omega$ or $C^s$. Since for each $\Q\in \F$ with $r_\Q=0$
\begin{align}
\int_\Q \eta \langle \Phi^\infty, \phi\star(\nu^\infty d\lambda^s\mres C^s)\rangle dx = 0
\end{align}
and $dist(\cup \F,\partial \Omega)>2t$, we get
\begin{align}
\int_\Omega \langle \Phi^\infty,\phi\star (\nu^\infty d\lambda^s\mres C^s)\rangle dx = & \sum_{\Q\in \F^s} \int_\Q \eta \langle \Phi^\infty,\phi\star (\nu^\infty d\lambda^s \mres C^s) \rangle dx + \mathcal{E}_2 \\
= & \sum_{\Q\in \F^s} \left( \int_\Q \eta dx \langle \Phi^\infty,\nu_{x_\Q}^\infty \rangle r_\Q + \mathcal{E}_3^\Q \right) + \mathcal{E}_2,
\end{align}
where $|\mathcal{E}_2| \leq \lambda^s(C^s\cap (\partial \Omega)_{2t})$. The third error is estimated in the following way:
\begin{align}
|\mathcal{E}_3^Q| \leq & \left| \int_\Q \left( \eta - \int_\Q \eta \right) \langle \Phi^\infty,\phi\star (\nu^\infty d\lambda^s\mres C^s)\rangle dx \right| \\
& \quad + \left| \int_\Q \eta \left( \int_\Q \langle \Phi^\infty,\phi\star (\nu^\infty d\lambda^s \mres C^s)\rangle dx - \langle \Phi^\infty,\nu_{x_\Q}^\infty \rangle r_\Q\right) \right| \\
\leq & \|\eta\|_{Lip} \ \L^n(\Q)^\frac{1}{n} \ \| \Phi^\infty \| \int_\Q \phi \star \lambda^s dx \\
&\quad + \|\eta\|_{Lip} \int_\Q \int_{C^s} \phi(x-y) \langle \Phi^\infty,\nu_y^\infty - \nu_{x_\Q}^\infty \rangle d\lambda^s(y)dx.
\end{align}
In particular, we obtain
\begin{align}
|\mathcal{E}_3^Q| \leq & t \int_\Q \phi\star \lambda^s dx + \int_\Q \int_{C^s} \phi(x-y) \omega(\|y-x_\Q\|) d\lambda^s(y)dx \\
\leq & (t+\omega(3t)) \int_\Q \phi\star \lambda^s dx.
\end{align}
From each $\Q\in \F^s$ we get
\begin{align}
\Phi(0) + \langle \Phi^\infty,\nu_{x_\Q}^\infty\rangle r_\Q = \int_\Q \Phi(r_\Q \overline{\nu^\infty_{x_\Q}} + D\phi^\Q)dx + \overbrace{\mathcal{E}_4^\Q}^{|\cdot|\leq \eps}.
\end{align}
Further computations show that
\begin{align}
\int_\Q |\Phi(r_\Q \overline{\nu^\infty_{x_\Q}} + D\phi^\Q) - \Phi(0)| dx \leq & \|\eps_{r_\Q \overline{\nu^\infty_{x_\Q}} + D\phi^\Q}\|_K \\
\leq & \| (\delta_0,\nu_{x_\Q}^\infty r_\Q) \|_K + \eps \leq 1 + r_\Q + \eps
\end{align}
and consequently
\begin{align}
\int_\Q \eta \int_\Q \Phi(r_\Q \overline{\nu^\infty_{x_\Q}} + D\phi^\Q)dx = \int_\Q \eta \Phi(r_\Q \overline{\nu^\infty_{x_\Q}} + D\phi^\Q)dx + \mathcal{E}_5^\Q,
\end{align}
where the error term is upper bounded by
\begin{align}
|\mathcal{E}_5^\Q| \leq \sup_\Q \left| \eta-\int_\Q \eta \right| \L^n(\Q) (1+r_\Q+\eps) \leq \L^n(\Q)^\frac{1}{n} \int_\Q (2+\phi\star\lambda^s)dx.
\end{align}
Finally, we estimate the last term with
\begin{align}
\int_\Q \eta \Phi(r_\Q \overline{\nu^\infty_{x_\Q}} + D\phi^\Q) dx = \int_\Q \eta \Phi(\phi\star \overline{\nu^\infty} d\lambda^s) + D\phi^\Q)dx + \mathcal{E}_6^\Q.
\end{align}
To bound the 6th error term we introduce an extra quantity
\begin{align}
\left| \int_\Q \eta \Big( \Phi(\phi\star \overline{\nu^\infty} d\lambda^s) + D\phi^\Q)dx - \Phi(\phi\star \overline{\nu^\infty} d\lambda^s\mres C^s) + D\phi^\Q)\Big) dx \right| \leq \int_\Q \phi\star( \lambda^s\mres \Omega\setminus C^s)dx,
\end{align}
and
\begin{align}
& \left| \int_\Q \eta \Big( \Phi(r_\Q \overline{\nu^\infty_{x_\Q}} + D\phi^\Q) dx - \Phi(\phi\star \overline{\nu^\infty} d\lambda^s\mres C^s) + D\phi^\Q)\Big)dx \right| \\
\leq & \int_\Q |r_\Q \overline{\nu^\infty_{x_\Q}} - \phi\star (\overline{\nu^\infty} d\lambda^s\mres C^s)|dx \\
\leq & \left| \int_\Q \phi\star \big( (\overline{\nu^\infty_{x_\Q}} - \overline{\nu^\infty} \big) d\lambda^s\mres C^s)dx \right| + \int_\Q \left| \fint_\Q \phi\star (\overline{\nu^\infty} d\lambda^s\mres C^s)dx' - \phi\star(\overline{\nu^\infty} d\lambda^s\mres C^s)\right| dx \\
\leq & \int_\Q \int_{C^s} \phi(x-y) |\overline{\nu^\infty_{x_\Q}} -\overline{\nu^\infty_y}| d\lambda^s dx + \int_\Q \left| \fint_\Q \phi\star (\overline{\nu^\infty} d\lambda^s \mres C^s)dx' - \phi\star (\overline{\nu^\infty} d\lambda^s\mres C^s)\right| dx \\
\leq & \omega(3t) \int_\Q \phi\star \lambda^s dx + \mathcal{E}_7^\Q.
\end{align}
The 7th error term is estimated by
\begin{align}
\mathcal{E}_7^\Q \leq & \int_\Q \fint_\Q \int_{C^s} \big| \phi(x'-y)-\phi(x-y) \big| |\overline{\nu^\infty_y}| d\lambda^s(y)dx'dx \\
\leq & \sqrt{n} \int_\Q \fint_\Q \int_{\Q+tQ} \|x-x'\| \int_0^1 \big| D\phi(x + (x'-x)\tau)(x'-x) \big| d\tau d\lambda^s(y)dx'dx.
\end{align}

Given the scaling of $\phi = \phi_t$ with respect to $t$, we have $|D\phi|\leq t^{-1} M (\chi_\Q)_t$, so the above can be bounded by
\begin{align}
\mathcal{E}_7^\Q \leq\sqrt{n} M \frac{ \L^n(\Q)^\frac{1}{n} }{t} \int_\Q (\chi_{2Q})_t \star \lambda^s dx.
\end{align}
For our choice of $a$ and $b$, we have that $\L^n(Q) = 2^{-n(a+b)}$. With $\xi^s$ defined above we obtain that
\begin{align}
\int_\Omega \eta \langle \Phi^\infty,\phi\star (\nu^\infty d\lambda^s)\rangle dx = \int_{\Cup \F^s} \eta \Phi(\xi^s) + \mathcal{E},
\end{align}
where
\begin{align}
|\mathcal{E}| \leq & \eps \lambda^s(\Omega) + (t+\omega(3t)) \int_{\cup \F^s} \Big( \phi\star \lambda^s + 2^{-a-b}(2+ \phi \star \lambda^s) \\
& \quad + \phi\star (\lambda^s \mres \Omega\setminus C^s) + \omega(3t)\ \phi\star \lambda^s dx + \sqrt{n} M 2^{-b} (\chi_{2Q})_t \star \lambda^s \Big)dx \\
\leq & \big( 2\eps + 2^{-a} + 2\omega(32^{-a}) + 2^{-a-b} + c_n M 2^{-b} \big) \lambda^s(\Omega) + 2^{1-a-b} \L^n(\Omega).
\end{align}
To conclude we add
\begin{align}
\int_{\Omega\setminus \cup \F^s} \eta \Phi(\xi^s)dx = \int_{\Omega\setminus \cup \F^s} \eta dx \Phi(0)
\end{align}
to both sides, and obtain
\begin{align}
\int_\Omega \eta \big( \Phi(0) + \langle \Phi^\infty,\phi\star (\nu^\infty d\lambda^s\mres \Omega)\rangle \big) dx = \int_\Omega \Phi(\xi^s)dx + \mathcal{E} + \int_{\cup \F^s} \eta dx \Phi(0).
\end{align}

Because $\cup \F^s \subset (C^s)_{2t}$ and since $\L^n(C^s)=0$ we can find $a_\eps\geq a, b_\eps\geq b$ such that 
\begin{align}
|\mathcal{E}| + \left| \int_{\cup \F^s} \eta dx \Phi(0)\right| \leq 3\eps (\L^n+\lambda^s)(\Omega).
\end{align}
The left-hand side tends to
\begin{align}
\int_\Omega \eta dx \Phi(0) + \int_\Omega \eta \langle \Phi^\infty,\nu_x^\infty\rangle d\lambda^s(x)
\end{align}
as $a\to \infty$, uniformly in $\eta$ and $\Phi$, and this concludes the proof.
\end{proof}

We now move on to the absolutely continuous part, which is proven similar and is a bit easier to construct.

\begin{lemma}
Let $\eps>0$, there is $t_\eps>0$ and $\psi_t\in \D(\Omega,\R^m)$ with $\|\psi_t\|_1 \leq \eps$ so that
\begin{align}
\left| \int_\Omega \eta( \langle \Phi,\nu_x\rangle + \lambda^a(x) \langle \Phi^\infty,\nu_x^\infty\rangle)dx - \int_\Omega \eta \Phi \big( \phi_t\star (\overline{\nu} + \overline{\nu^\infty} \lambda^a(x))dx\mres \Omega + D\psi_t \big) dx \right| < \eps
\end{align}
holds for $t\in (0,t_\eps)$, uniformly in $\|\eta\|_{Lip} \leq 1$ and $\|T\Phi\|_{Lip(e_{f_i,i\in\N})} \leq 1$.
\end{lemma}

\begin{proof}
Fix $\eps\in (0,1)$ and apply Luzin's theorem to the $\L^n$-measurable map
\begin{align}
\Omega\ni x \mapsto (\nu_x,\lambda^a(x) \nu_x^\infty) \in \M_1^+(\R^d)\times \M^+(\partial e_{f_i,i\in\N}) \hookrightarrow \big( (T^{-1}e_{f_i,i\in \N})^* \big)^+
\end{align}
to find a compact set $C^a\subset \Omega$ such that
\begin{align}
\int_{\Omega\setminus C^a} M(x)dx < \eps, \ M(x) =  \langle \nu_x,|\cdot|\rangle + \lambda^a(x),
\end{align}
and $\omega$ be the modulus of continuity over $C^a$, i.e.
\begin{align}
\|(\nu_x,\lambda^a(x)\nu_x^\infty) - (\nu_y,\lambda^a(y)\nu_y^\infty)\|_K \leq \omega(\|x-y\|) \quad \text{ for all } x,y\in C^a.
\end{align}
Fix $d\in \N$ and $s\in (0,1)$ and let $\F^a$ be the family of dyadic cubes in $\R^n$ of side length $t=2^{-d}$, i.e.
\begin{equation}
\F^a = \{ \Q\in \D_d:d(\Q,\partial \Omega)>t,\ \L^n(\Q\cap C^a)>s\L^n(\Q)\},
\end{equation}
where the distance is induced by $\|\cdot\|_\infty$ over vectors in $\R^n$, and $d$ and $s$ will be selected later in the proof. For every $\Q\in \F^a$ select $x_\Q$. By \refp{result_homogeneous} we have $\psi^Q \in \D(Q,\R^{n\times m})$ with $\|\psi^\Q\|_1<\eps \frac{\L^n(\Q)}{\L^n(\Omega)}$ and
\begin{align}
\| (\nu_{x_\Q},\lambda^a(x_\Q)\nu_{x_\Q}^\infty) - \eps_{\overline{\nu_{x_\Q}} +\lambda^a(x_\Q)\overline{\nu^\infty_{x_\Q}} + D\psi^\infty} \|_K < \eps.
\end{align}
Let $\psi = \sum_\Q \psi^\Q \in \D(\Omega,\R^{m\times n})$ and $\|\psi\|_1\leq \eps$. Then
\begin{align}
& \int_\Omega \eta(\langle \Psi,\nu_x\rangle + \lambda^a(x)\langle \Psi^\infty,\nu_x^\infty\rangle)dx = \sum_{\Q\in\F^a} \int_\Q \eta(\langle \Psi,\nu_x\rangle + \lambda^a(x)\langle \Psi^\infty,\nu_x^\infty\rangle)dx + \mathcal{E}_1 \\
&\text{with } |\mathcal{E}_1| \leq \int_{\Omega\setminus \cup \F^a} M(x)dx \leq \eps
\end{align}
for large enough $d$. Next
\begin{align}
\sum_{\Q\in \F^a} \int_\Q \eta(\langle \Psi,\nu_x\rangle + \lambda^a(x)\langle \Psi^\infty,\nu_x^\infty\rangle)dx = \sum_{\Q\in \F^a} \fint_\Q \eta \int_\Q (\langle \Psi,\nu_x\rangle + \lambda^a(x)\langle \Psi^\infty,\nu_x^\infty\rangle)dx + \mathcal{E}_2,
\end{align}
where
\begin{align}
|\mathcal{E}_2| \leq t \sum_{\Q\in \F^a} \int_\Q |\langle \Phi,\nu_x\rangle + \lambda^a(x) \langle \Phi^\infty,\nu_x^\infty\rangle | dx \leq t \int_\Omega M(x)dx.
\end{align}
We further estimate, on every set $\Q\in \F^a$,
\begin{align}
& \left| \int_\Q \langle \Phi,\nu_x\rangle + \lambda^a(x) \langle \Phi^\infty,\nu_x^\infty\rangle - \L^n(\Q) \Big( \langle \Phi,\nu_{x_\Q}\rangle + \lambda^a(x_\Q) \langle \Phi^\infty,\nu_{x_\Q}^\infty\rangle \Big) \right| \\
\leq & \frac{\omega(t)}{s}\L^n(\Q) + \frac{1-s}{s} \int_\Q |\langle \Phi,\nu_x\rangle + \lambda^a(x) \langle \Phi^\infty,\nu_x^\infty\rangle|dx + \int_{\Q\setminus C^a} |\langle \Phi,\nu_x\rangle + \lambda^a(x) \langle \Phi^\infty,\nu_x^\infty\rangle|dx.
\end{align}
By the linear growth of $f$, we get that
\begin{align}
\sum_{\Q\in \F^a} \fint_\Q \eta \int_\Q \langle \Phi,\nu_x\rangle + \lambda^a(x) \langle \Phi^\infty,\nu_x^\infty\rangle dx = \sum_{\Q\in \F^a} \Big( \langle \Phi,\nu_{x_\Q}\rangle + \lambda^a(x_\Q) \langle \Phi^\infty,\nu_{x_\Q}^\infty\rangle \Big) \int_\Q \eta + \mathcal{E}_3
\end{align}
where
\begin{align}
|\mathcal{E}_3| \leq \frac{\omega(t)}{s}\L^n(\Q) + \frac{1-s}{s} \int_\Q M(x) dx + \eps.
\end{align}
For every $\Q\in \F^a$ we have that
\begin{align}
f(x_\Q) = \fint_\Q \Phi(\overline{\nu_{x_\Q}} + \lambda^a(x_\Q) \overline{\nu^\infty_{x_\Q}} + D\phi^\Q) + \overbrace{\mathcal{E}_4^\Q}^{|\cdot| \leq \eps}.
\end{align}
Set $v^a(x) = \overline{\nu_x} + \lambda^a(x) \overline{\nu^\infty_x}$. Letting $\Phi = z\cdot e_i$, $(e_i)$ canonical basis of $\R^{m\times n}$, we obtain, from continuity over $C^a$, $|v^a-v^a(x_\Q)| \leq \omega(t)$ on $\Q\cap C^a$ for each $\Q\in \F^a$. Consequently,
\begin{align}
\int_\Q |v^a-v^a(x_\Q)| dx \leq \frac{\omega(t)}{s} \L^n(\Q) + \int_{\Q\setminus C^a} |v^a|dx + \frac{1-s}{s} \int_\Q |v^a|dx
\end{align}
for all $\Q\in \F^a$. Because $|v^a|\leq M(x)$ and $Lip(\Phi) \leq 5$, then
\begin{align}
\sum_{\Q\in \F^a} \fint_\Q \eta dx \int_\Q \Phi(v^a(x_\Q) + D\psi^\Q)dx = \sum_{\Q\in \F^a} \fint_\Q \eta dx \int_\Q \Phi(v^a + D\psi^Q)dx + \mathcal{E}_5
\end{align}
with
\begin{align}
|\mathcal{E}_5| \leq 5 \frac{\omega(t)}{s} \L^n(Q) + 5\eps + 5\frac{1-s}{s} \int_\Omega M(x)dx.
\end{align}
Combining some of the previous estimates, we get
\begin{align}
\sum_{\Q\in \F^a} \fint_\Q \eta dx \int_\Q \Phi(v^a + D\psi^\Q)dx = \sum_{\Q\in \F^a} \int_\Q \eta \Phi(v^a + D\psi^\Q)dx + \mathcal{E}_6,
\end{align}
and
\begin{align}
|\mathcal{E}_6| \leq & t \sum_{\Q\in \F^a} \int_\Q |\Phi(v^a + D\psi^\Q)| dx \leq t |\mathcal{E}_5| + t \sum_{\Q\in \F^a} \int_\Q |\Phi(v^a(x_\Q) + D\psi^\Q)| dx \\
\leq & t|\mathcal{E}_5| + t\eps \L^n(\Omega) + t \sum_{\Q\in \F^a} \L^n(\Q) (\langle |\Phi|,\nu_{x_\Q}\rangle + \lambda^a(x_\Q) \langle |\Phi|^\infty,\nu_{x_\Q}^\infty\rangle) \\
\leq & t|\mathcal{E}_5| + t\eps \L^n(\Omega) + t \frac{\omega(t)}{s}\L^n(|) + \left( \int_{Q\setminus C^a} + \frac{1-s}{s}\int_\Omega \right) M(x) dx.
\end{align}
Finally, if $\phi_t$ is a standard mollifier, then $\phi_t \star v^a\mres \Omega \xrightarrow{L^1(\Omega)} v^a$ as $t\to 0$, and so
\begin{align}
\int_\Omega \eta \Phi(v^a+D\psi)dx = \int_\Omega \eta \Phi(\phi_t \star v^a\mres \Omega + D\psi)dx + \mathcal{E}_7,
\end{align}
where using again that $\Phi$ is Lipschitz over $\R^{m\times n}$,
\begin{align}
|\mathcal{E}_7| \leq Lip(\Phi) \int_\Omega |\phi_t \star v^a\mres \Omega - v^a|dx \leq 5 \int_\Omega |\phi_t\star v^a\mres \Omega - v^a|dx.
\end{align}
This concludes the proof.
\end{proof}

\newpage
\appendix

\section{Appendix}

\begin{theorem}[Vitali convergence theorem] \label{Appendix:Vitali_conv}
Let $\mu\in \M^+(\Omega)$ and $f_n,f\in L^1(\Omega,\mu)$. Then $f_n\to f$ in $L^1(\Omega,\mu)$ if and only if $f_n\to f$ in measure and $f_n$ is uniformly integrable.
\end{theorem}

\begin{proof}
See \citep{bogachev2007measure}{268}, theorem 4.5.4.
\end{proof}

\begin{theorem}[Stone-Weierstrass]
Let $X$ be a compact Hausdorff space. If $\A$ is a closed subalgebra of $C(X)$ that separates points, then either $\A=C(X)$ or there is $x_0\in X$ so that $\A = \{f\in C(X):f(x_0)=0\}$. In particular, $\A=C(X)$ if and only if $\A$ contains the constant functions.
\end{theorem}

\begin{proof}
See \citep{folland1999real}{139}, theorem 4.45.
\end{proof}

\begin{theorem}[Tychonoff] \label{Appendix:Tychonoff}
If $\{X_\alpha\}_{\alpha\in A}$ is a family of compact spaces, $\Pi_{\alpha\in A} X_\alpha$ is compact in the product topology.
\end{theorem}

\begin{theorem}[Banach-Alaoglu sequential version] \label{Appendix:Banach-Alaoglu}
Let $X$ be a separable Banach space and $\B\subset X^*$ the closed unit ball of the dual. Then $\B$ is weakly* sequentially compact.
\end{theorem}

\begin{proof}
See \citep{lax2014functional}{107}, theorem 12.
\end{proof}

\begin{lemma}[Chacon biting lemma] \label{Appendix:Chacon}
Let $\mu\in \M^+(\Omega)$ and $v_j\in L^1(\Omega,\mu)$ be a sequence such that $\sup_j \|v_j\|<\infty$. There are sets $E_k\subset E_{k+1}, \mu(\Omega\setminus E_k)\to 0$ and a subsequence $(v_{j_i})_i$ of $(v_j)_j$ and $v\in L^1(\mu)$ so that $v_{j_i} \rightharpoonup v$ in $L^1(E_k,\mu)$ for all $k$.
\end{lemma}

\begin{proof}
See \cite{ball1989remarks}.
\end{proof}

\begin{lemma}[Kantorovich metric] \label{Appendix:Kantorovich}
Let $(X,d)$ be a metric space. The Kantorich metric on $\M^+(X)$ generates the same topology as the weak* topology of measures.
\end{lemma}

\begin{proof}
See \cite{kristensen2019oscillation}.
\end{proof}

\begin{lemma} \label{Appendix:area-strict}
Let $\eta\in \M(\Omega,\R^d)$, $\Omega\subset \R^n$ open bounded set, and $(\phi_\eps)_{0<\eps\leq 1}$ be a family of standard mollifiers, $supp(\phi_1)\subset \B$. Then
\begin{align}
\eta\star \phi_\eps \xrightarrow{\text{area-strictly}} \eta.
\end{align}
\end{lemma}

\begin{proof}
Because $(\L^n\mres \Omega,\eta_\eps)\rightharpoonup^* (\L^n\mres\Omega,\eta)$, we immediately obtain the lower bound
\begin{align}
\liminf_{\eps\to 0} |(\L^n,\rho_\eps)|(\Omega) \geq |(\L^n,\rho)|(\Omega).
\end{align}
To prove the upper semi-continuity of the above quantity, notice the following equality:
\begin{align}
(\L^n\mres \Omega,\rho_\eps) = (\L^n\mres \Omega,\rho)\star \phi_\eps - (\L^n\mres \Omega_{-\eps} \star(\delta_0-\phi_\eps),0 ),
\end{align}
where
\begin{align}
\Omega_{-\eps} = \{x\in \Omega:d(x,\partial \Omega)>\eps\}.
\end{align}
Using then Jensen's inequality
\begin{align}
\limsup_{\eps\to 0} |(\L^n,\rho_\eps)|(\Omega) \leq & \limsup_{\eps\to 0} |(\L^n\mres \Omega,\rho)\star \phi_\eps|(\Omega) + | (\L^n\mres \Omega_{-\eps} \star(\delta_0-\phi_\eps),0 ) |(\Omega) \\
\leq & |(\L^n\mres \Omega,\rho)|(\Omega) + \limsup_{\eps\to 0} 2\L^n(\Omega\setminus\Omega_{-\eps}) = |(\L^n,\rho)|(\Omega)
\end{align}
\end{proof}

\begin{theorem}[Atomic decomposition] \label{Appendix:atomic_decomposition}
Let $\mu\in \M(\Omega)$. Then there exists a purely atomic measure $\mu^a$ and a non-atomic measure $\mu^{n-a}$ such that $\mu = \mu^a + \mu^{n-a}$.
\end{theorem}

\begin{proof}
See \citep{fonseca2007modern}{13}.
\end{proof}

\newpage
\bibliographystyle{alpha}
\bibliography{biblio}

\begin{thebibliography}{ADM92}

\bibitem[AB97]{alibert1997non}
Jean-Jacques Alibert and Guy Bouchitt{\'e}.
\newblock Non-uniform integrability and generalized young measure.
\newblock {\em Journal of Convex Analysis}, 4:129--148, 1997.

\bibitem[ADM92]{ambrosio1992relaxation}
Luigi Ambrosio and Gianni Dal~Maso.
\newblock On the relaxation in bv (Ω; rm) of quasi-convex integrals.
\newblock {\em Journal of functional analysis}, 109(1):76--97, 1992.

\bibitem[AF84]{acerbi1984semicontinuity}
Emilio Acerbi and Nicola Fusco.
\newblock Semicontinuity problems in the calculus of variations.
\newblock {\em Archive for Rational Mechanics and Analysis}, 86(2):125--145,
  1984.

\bibitem[AFP00]{ambrosio2000functions}
Luigi Ambrosio, Nicola Fusco, and Diego Pallara.
\newblock {\em Functions of bounded variation and free discontinuity problems}.
\newblock Courier Corporation, 2000.

\bibitem[Bal89]{ball1989version}
John~M Ball.
\newblock A version of the fundamental theorem for young measures.
\newblock In {\em PDEs and continuum models of phase transitions}, pages
  207--215. Springer, 1989.

\bibitem[BL73]{berliocchi1973integrandes}
Henri Berliocchi and Jean-Michel Lasry.
\newblock Int{\'e}grandes normales et mesures param{\'e}tr{\'e}es en calcul des
  variations.
\newblock {\em Bulletin de la Soci{\'e}t{\'e} Math{\'e}matique de France},
  101:129--184, 1973.

\bibitem[BM89]{ball1989remarks}
John~M Ball and F~Murat.
\newblock Remarks on chacon’s biting lemma.
\newblock {\em Proceedings of the American Mathematical Society},
  107(3):655--663, 1989.

\bibitem[BR07]{bogachev2007measure}
Vladimir~Igorevich Bogachev and Maria Aparecida~Soares Ruas.
\newblock {\em Measure theory}, volume~1.
\newblock Springer, 2007.

\bibitem[BZ90]{ball1990lower}
JM~Ball and K-W Zhang.
\newblock Lower semicontinuity of multiple integrals and the biting lemma.
\newblock {\em Proceedings of the Royal Society of Edinburgh Section A:
  Mathematics}, 114(3-4):367--379, 1990.

\bibitem[CC76]{chandler1976hausdorff}
R.E. Chandler and R.E. Chandler.
\newblock {\em Hausdorff Compactifications}.
\newblock Lecture notes in pure and applied mathematics. M. Dekker, 1976.

\bibitem[Dac07]{dacorogna2007direct}
Bernard Dacorogna.
\newblock {\em Direct methods in the calculus of variations}, volume~78.
\newblock Springer Science \& Business Media, 2007.

\bibitem[DM87]{diperna1987oscillations}
Ronald~J DiPerna and Andrew~J Majda.
\newblock Oscillations and concentrations in weak solutions of the
  incompressible fluid equations.
\newblock {\em Communications in mathematical physics}, 108(4):667--689, 1987.

\bibitem[FK10]{fonseca2010oscillations}
Irene Fonseca and Martin Kru{\v{z}}{\'\i}k.
\newblock Oscillations and concentrations generated by-free mappings and weak
  lower semicontinuity of integral functionals.
\newblock {\em ESAIM: Control, Optimisation and Calculus of Variations},
  16(2):472--502, 2010.

\bibitem[FL07]{fonseca2007modern}
Irene Fonseca and Giovanni Leoni.
\newblock {\em Modern Methods in the Calculus of Variations: L\^{} p Spaces}.
\newblock Springer Science \& Business Media, 2007.

\bibitem[Fol99]{folland1999real}
Gerald~B Folland.
\newblock {\em Real analysis: modern techniques and their applications}.
\newblock Wiley, 1999.

\bibitem[KK16]{kirchheim2016rank}
Bernd Kirchheim and Jan Kristensen.
\newblock On rank one convex functions that are homogeneous of degree one.
\newblock {\em Archive for rational mechanics and analysis}, 221(1):527--558,
  2016.

\bibitem[KP91]{kinderlehrer1991characterizations}
David Kinderlehrer and Pablo Pedregal.
\newblock Characterizations of young measures generated by gradients.
\newblock {\em Archive for rational mechanics and analysis}, 115(4):329--365,
  1991.

\bibitem[KP94]{kinderlehrer1994gradient}
David Kinderlehrer and Pablo Pedregal.
\newblock Gradient young measures generated by sequences in sobolev spaces.
\newblock {\em The Journal of Geometric Analysis}, 4(1):59, 1994.

\bibitem[KR96]{kruvzik1996explicit}
Martin Kru{\v{z}}{\'\i}k and Tom{\'a}{\v{s}} Roub{\'\i}{\v{c}}ek.
\newblock Explicit characterization oflp-young measures.
\newblock {\em Journal of mathematical analysis and applications},
  198(3):830--843, 1996.

\bibitem[KR10a]{kristensen2010characterization}
Jan Kristensen and Filip Rindler.
\newblock Characterization of generalized gradient young measures generated by
  sequences in w1, 1 and bv.
\newblock {\em Archive for rational mechanics and analysis}, 197(2):539--598,
  2010.

\bibitem[KR10b]{kristensen2010relaxation}
Jan Kristensen and Filip Rindler.
\newblock Relaxation of signed integral functionals in bv.
\newblock {\em Calculus of Variations and Partial Differential Equations},
  37(1-2):29--62, 2010.

\bibitem[KR19]{kristensen2019oscillation}
Jan Kristensen and Bogdan Rai{\c{t}}{\u{a}}.
\newblock Oscillation and concentration in sequences of pde constrained
  measures.
\newblock {\em arXiv preprint arXiv:1912.09190}, 2019.

\bibitem[Kri15]{kristensen2015young}
Jan Kristensen.
\newblock Lecture notes on young measures, 2015.

\bibitem[Lax14]{lax2014functional}
Peter~D Lax.
\newblock {\em Functional analysis}.
\newblock John Wiley \& Sons, 2014.

\bibitem[Mor52]{morrey1952}
Charles~B. Morrey.
\newblock Quasi-convexity and the lower semicontinuity of multiple integrals.
\newblock {\em Pacific J. Math.}, 2(1):25--53, 1952.

\bibitem[M{\"u}l92]{muller1992quasiconvex}
Stefan M{\"u}ller.
\newblock On quasiconvex functions which are homogeneous of degree 1.
\newblock {\em Indiana University mathematics journal}, pages 295--301, 1992.

\bibitem[Res68]{reshetnyak1968weak}
Yu~G Reshetnyak.
\newblock Weak convergence of completely additive vector functions on a set.
\newblock {\em Siberian Mathematical Journal}, 9(6):1039--1045, 1968.

\bibitem[Rin18]{rindler2018calculus}
Filip Rindler.
\newblock {\em Calculus of variations}, volume~5.
\newblock Springer, 2018.

\bibitem[You37]{young1937generalized}
Laurence~Chisholm Young.
\newblock Generalized curves and the existence of an attained absolute minimum
  in the calculus of variations.
\newblock {\em Comptes Rendus de la Societe des Sci. et des Lettres de
  Varsovie}, 30:212--234, 1937.

\end{thebibliography}

\end{document}